\documentclass[reqno]{amsart}

\usepackage{amssymb}
\usepackage{graphicx}
\usepackage{amscd}
\usepackage[pagebackref]{hyperref}
\usepackage{color}
\usepackage{tabularx}
\usepackage[table]{xcolor}
\usepackage{float}
\usepackage{graphics,amsmath,amssymb}
\usepackage{amsthm}
\usepackage{amsfonts}
\usepackage{latexsym}
\usepackage{epsf}
\usepackage{xifthen}
\usepackage{mathrsfs}
\usepackage{dsfont}
\usepackage{makecell}
\usepackage{subfig}
\usepackage{amsmath}
\usepackage{listings}
\usepackage{etoolbox}
\usepackage{fancyhdr}
\usepackage{pdflscape}
\usepackage[title,toc,titletoc]{appendix}
\usepackage{enumitem}
\usepackage[noadjust]{cite}
\usepackage{tikz}
\usetikzlibrary{automata,positioning,arrows}
\usepackage{young}
\usepackage[object=vectorian]{pgfornament} 
\usepackage{lipsum,tikz}
\usepackage{multirow}
\usepackage[OT2,T1]{fontenc}
\usepackage{mathtools}
\usepackage{ytableau}

\hypersetup{
	colorlinks=true, 
	linktoc=all, 
	linkcolor=blue} 

\numberwithin{equation}{section}

\theoremstyle{theorem}
\newtheorem{theorem}{Theorem}[section]
\newtheorem*{theorem*}{Theorem}

\newtheorem{corollary}[theorem]{Corollary}
\newtheorem{lemma}[theorem]{Lemma}
\newtheorem{proposition}[theorem]{Proposition}

\newtheorem{innercustomgeneric}{\customgenericname}
\providecommand{\customgenericname}{}
\newcommand{\newcustomtheorem}[2]{%
	\newenvironment{#1}[1]
	{%
		\renewcommand\customgenericname{#2}%
		\renewcommand\theinnercustomgeneric{##1}%
		\innercustomgeneric
	}
	{\endinnercustomgeneric}
}
\newcustomtheorem{ctheorem}{Theorem}
\newcustomtheorem{clemma}{Lemma}

\theoremstyle{definition}

\newtheorem*{example*}{Example}
\newtheorem*{examples*}{Examples}
\newtheorem{remark}[theorem]{Remark}
\newtheorem*{remark*}{Remark}
\newtheorem*{remarks*}{Remarks}
\newtheorem*{note*}{Note}

\newtheoremstyle{named}{}{}{\itshape}{}{\bfseries}{.}{.5em}{\thmnote{#3} #1}
\theoremstyle{named}

\newtheoremstyle{customized}{}{}{\itshape}{}{\bfseries}{.}{.5em}{\thmnote{#3}}
\theoremstyle{customized}

\newcommand{\arxiv}[1]{\href{http://arxiv.org/abs/#1}{arXiv:#1}}

\DeclareMathAlphabet{\mydutchcal}{U}{dutchcal}{m}{n}
\DeclareMathAlphabet{\myrsfso}{U}{rsfso}{m}{n}

\newcommand{\cM}{\myrsfso{M}}

\newcommand{\dcP}{\mydutchcal{P}}

\newcommand{\bs}{\boldsymbol{s}}

\newcommand{\Cat}{\operatorname{Cat}}
\newcommand{\ddd}{\operatorname{d}}
\newcommand{\sgn}{\operatorname{sgn}}

\newcommand{\rU}{\mathrm{U}}
\newcommand{\rD}{\mathrm{D}}
\newcommand{\rL}{\mathrm{L}}

\newcommand{\LHS}{\operatorname{LHS}}

\title[Hankel determinants of Catalan-like sequences]{Hankel determinants of Catalan-like sequences}

\author[S. Chern]{Shane Chern}
\address[S. Chern]{Fakult\"at f\"ur Mathematik, Universit\"at Wien, Oskar-Morgenstern-Platz 1, Wien 1090, Austria}
\email{chenxiaohang92@gmail.com, xiaohangc92@univie.ac.at}

\author[W. Shi]{Wenle Shi}
\address[W. Shi]{School of Mathematical Sciences, Dalian University of Technology, Dalian 116024, P.R. China}
\email{shi-wenle@hotmail.com}

\date{}



\begin{document}
	
\sloppy

\begin{abstract}
	In this paper, we compute the (shifted) Hankel determinants of Catalan-like sequences, which arise naturally from the weighted enumerations of nonintersecting Motzkin meanders. Among these determinant evaluations, one and a half are newly discovered, featuring generic shifted Hankel determinants; two were formulated earlier by Cigler and Krattenthaler in an equivalent combinatorial form; and the rest were conjectured by Cigler. As an application, we further confirm a conjectural binomial determinant identity proposed by Cigler and Krattenthaler.
\end{abstract}

\maketitle

\section{Introduction}

For a sequence $(u_{n})_{n\ge 0}$, its $n$-th \emph{Hankel determinant} is
\begin{align*}
	\det_{0\leq i,j \leq N-1}(u_{i+j})=\det\begin{pmatrix}u_{0} & u_{1} & \cdots & u_{N-1}\\
		u_{1} & u_{2} & \cdots & u_{N}\\
		\vdots & \vdots & \ddots & \vdots\\
		u_{N-1} & u_{N} & \cdots & u_{2N-2}
	\end{pmatrix}.
\end{align*}
In addition, given a fixed nonnegative integer $m$, by replacing the sequence $(u_{n})_{n\ge 0}$ with $(u_{n+m})_{n\ge 0}$, we have the \emph{$m$-shifted Hankel determinants} $\det_{0\leq i,j \leq N-1}(u_{i+j+m})$. These determinants are closely related the theory of orthogonal polynomials; see Ismail's monograph~\cite[Chapter~2]{Ism2005} or Krattenthaler's surveys~\cite[Section~2.7]{Kra1999} and \cite[Section~5.4]{Kra2005}. There is a rich history in the study of Hankel determinants of classic sequences such as Bernoulli and Euler numbers \cite{AC1959}, their standard $q$-analogs \cite{CZ2017,CJ2024}, and other variants \cite{AW2002,DJ2021,DJ2022,Mil2002}.

Let $\bs=(s_k)_{k\geq0}$ be a fixed sequence. We define a two-dimensional array of numbers $\Cat_{n,k}^{(\bs)}$ by the recurrence for $n\ge 1$,
\begin{equation}\label{eq:a-rec}
	\Cat_{n,k}^{(\bs)}
	=\Cat_{n-1,k-1}^{(\bs)}+s_k \Cat_{n-1,k}^{(\bs)}
	+\Cat_{n-1,k+1}^{(\bs)},
\end{equation}
where the boundary values are
\begin{align}\label{eq:a-boundary}
	\Cat_{0,0}^{(\bs)}=1,\qquad \Cat_{0,k}^{(\bs)}=0\quad(k\geq1),\qquad \Cat_{n,k}^{(\bs)}=0\quad(k<0).
\end{align}
Aigner~\cite[p.~36]{Aig1999} called the numbers $\Cat_{n,0}^{(\bs)}$ \emph{Catalan-like numbers of type $\bs$}. This is because if we specialize the sequence $\bs$ to $(1,2,2,\ldots)$, then as exhibited in \cite[p.~36, Example~1]{Aig1999},
\begin{align*}
	\Cat_{n,0}^{(1,2,2,\ldots)} = \frac{1}{n+1}\binom{2n}{n},
\end{align*}
the $n$-th \emph{Catalan number}. We also have other important specializations:
\begin{itemize}[itemindent=*, leftmargin=*, itemsep=0pt, topsep=2pt]
	\item The \emph{Riordan numbers}~\cite[A005043]{OEIS} are given by $\Cat_{n,0}^{(0,1,1,\ldots)}$;
	
	\item The \emph{Fine numbers}~\cite[A000957]{OEIS} are given by $\Cat_{n,0}^{(0,2,2,\ldots)}$;
	
	\item The \emph{Motzkin numbers}~\cite[A001006]{OEIS} are given by $\Cat_{n,0}^{(1,1,1,\ldots)}$.
\end{itemize}
Aigner~\cite[p.~45, Proposition~6]{Aig1999} then showed that
\begin{align*}
	\det_{0\leq i,j \leq N-1}\big(\Cat_{i+j,0}^{(\bs)}\big) = 1,
\end{align*}
for any choice of $\bs$. Following this, the evaluations of shifted Hankel determinants for Catalan-like numbers were established in \cite{LM2019,MW2017,MWY2017}.

In this work, we treat all sequences $(\Cat_{n,k}^{(\bs)})_{n\ge 0}$ with $k$ fixed as \emph{Catalan-like sequences}, and we are going to evaluate their Hankel determinants at the specializations
\begin{align*}
	\bs = (b+c,c,c,\ldots), \qquad\qquad b\in\{0,1\},
\end{align*}
in which case we abbreviate
\begin{align*}
	\Cat_{n,k}^{(b,c)} := \Cat_{n,k}^{(b+c,c,c,\ldots)}.
\end{align*}

We start by noting that, in \cite[p.~146, Theorems~1 and 2]{CK2011}, Cigler and Krattenthaler evaluated the Hankel determinants for a weighted counting function for Motzkin meanders. As we will explain in Section~\ref{sec:comb}, their results reduce to the following two identities in our setting with $b=0$.

\begin{theorem}[Cigler--Krattenthaler]\label{th:b0-m0}
	Let $k\ge 0$ and write $N=(k+1)n+r$ with $0\le r\le k$. Then for $c$ an indeterminate,
	\begin{align}
		\det_{0\leq i,j \leq N-1}\big(\Cat_{i+j,k}^{(0,c)}\big) = \begin{cases}
			(-1)^{\binom{k+1}{2}n}, & \text{if $r=0$},\\
			0, & \text{otherwise}.
		\end{cases}
	\end{align}
\end{theorem}

\begin{theorem}[Cigler--Krattenthaler]\label{th:b0-m1}
	Let $k\ge 0$ and write $N=(k+1)n+r$ with $0\le r\le k$. Then for $c$ an indeterminate,
	\begin{align}
		\det_{0\leq i,j \leq N-1}\big(\Cat_{i+j+1,k}^{(0,c)}\big) = \begin{cases}
			(-1)^{\binom{k+1}{2}n} \tfrac{F_{(k+1)(n+1)}(c)}{F_{k+1}(c)}, & \text{if $r=0$},\\[6pt]
			(-1)^{\binom{k+1}{2}n + \binom{k}{2}} \tfrac{F_{(k+1)(n+1)}(c)}{F_{k+1}(c)}, & \text{if $r=k$},\\[6pt]
			0, & \text{otherwise}.
		\end{cases}
	\end{align}
\end{theorem}

In Theorem~\ref{th:b0-m1}, we have used a variant of Fibonacci polynomials defined recursively by
\begin{align}\label{eq:F-rec}
	F_n(x) = xF_{n-1}(x) - F_{n-2}(x),\qquad\qquad F_0(x) = 0,\quad F_1(x)=1.
\end{align}
In the same vein, we also introduce a variant of Lucas polynomials given by
\begin{align}\label{eq:L-rec}
	L_n(x) = xL_{n-1}(x) - L_{n-2}(x),\qquad\qquad L_0(x) = 2,\quad L_1(x)=x.
\end{align}

Following Theorems~\ref{th:b0-m0} and \ref{th:b0-m1}, Cigler~\cite{Cig2023} recently made further conjectural Hankel determinant evaluations. The main objective of this paper is to confirm these conjectures and work out several newly discovered generic shifted Hankel determinants of Catalan-like sequences.

We start with the case where $b=0$.

\begin{theorem}[Cigler~\cite{Cig2023}, Conjecture~5]\label{th:b0-m2}
	Let $k\ge 1$ and write $N=(k+1)n+r$ with $0\le r\le k$. Then for $c$ an indeterminate,
	\begin{align}
		&\!\!\!\!\det_{0\leq i,j \leq N-1}\big(\Cat_{i+j+2,k}^{(0,c)}\big)\notag\\
		&\quad = \begin{cases}
			(-1)^{\binom{k+1}{2}n} \left(\tfrac{F_{(k+1)(n+1)}(c)}{F_{k+1}(c)}\right)^2, & \text{if $r=0$},\\[6pt]
			(-1)^{\binom{k+1}{2}n + \binom{k-1}{2}} \left(\tfrac{F_{(k+1)(n+1)}(c)}{F_{k+1}(c)}\right)^2, & \text{if $r=k-1$},\\[6pt]
			(-1)^{\binom{k+1}{2}n + \binom{k}{2}}(k+1) F_{k+1}(c) \sum\limits_{j=1}^{n+1} \left(\tfrac{F_{(k+1)j}(c)}{F_{k+1}(c)}\right)^2, & \text{if $r=k$},\\[6pt]
			0, & \text{otherwise}.
		\end{cases}
	\end{align}
	In addition, the $r=k$ case also holds when $k=0$. Namely,
	\begin{align}
		\det_{0\leq i,j \leq n-1}\big(\Cat_{i+j+2,0}^{(0,c)}\big) = \sum\limits_{j=1}^{n+1} \big(F_{j}(c)\big)^2.
	\end{align}
\end{theorem}

For generic $m$, we have the following result covering two residue classes. In particular, the $r=0$ case was conjectured by Cigler~\cite[Conjecture~6]{Cig2023}.

\begin{theorem}\label{th:b0-m}
	Let $1\le m\le k+1$ and write $N=(k+1)n+r$ with $0\le r\le k$. Then for $c$ an indeterminate,
	\begin{align}
		&\!\!\!\!\det_{0\leq i,j \leq N-1}\big(\Cat_{i+j+m,k}^{(0,c)}\big)\notag\\
		&\quad = \begin{cases}
			(-1)^{\binom{k+1}{2}n} \left(\tfrac{F_{(k+1)(n+1)}(c)}{F_{k+1}(c)}\right)^m, & \text{if $r=0$},\\[6pt]
			(-1)^{\binom{k+1}{2}n+\binom{k-m+1}{2}} \left(\tfrac{F_{(k+1)(n+1)}(c)}{F_{k+1}(c)}\right)^m, & \text{if $r=k-m+1$}.
		\end{cases}
	\end{align}
\end{theorem}

\begin{remark}
	The original conjectures of Cigler in Theorems~\ref{th:b0-m2} and \ref{th:b0-m} were stated in a different form. However, they can be recovered by the relation \eqref{eq:F-L}. The same argument also applies to the results below.
\end{remark}

We continue with the case where $b=1$.

\begin{theorem}[Cigler~\cite{Cig2023}, Conjecture~7]\label{th:b1-m0}
	Let $k\ge 0$ and write $N=(2k+1)n+r$ with $0\le r\le 2k$. Then for $c$ an indeterminate,
	\begin{align}
		\det_{0\leq i,j \leq N-1}\big(\Cat_{i+j,k}^{(1,c)}\big) = \begin{cases}
			1, & \text{if $r=0$},\\
			(-1)^{\binom{k+1}{2}}, & \text{if $r=k+1$},\\
			0, & \text{otherwise}.
		\end{cases}
	\end{align}
\end{theorem}

\begin{theorem}[Cigler~\cite{Cig2023}, Conjecture~8]\label{th:b1-m1}
	Let $k\ge 1$ and write $N=(2k+1)n+r$ with $0\le r\le 2k$. Then,
	\begin{align}
		&\!\!\!\!\det_{0\leq i,j \leq N-1}\big(\Cat_{i+j+1,k}^{(1,c)}\big)\notag\\
		&\quad = \begin{cases}
			\frac{F_{(2k+1)n}(c)}{F_{2k+1}(c)} + \frac{F_{(2k+1)(n+1)}(c)}{F_{2k+1}(c)}, & \text{if $r=0$},\\[6pt]
			(-1)^{\binom{k}{2}}\left(\frac{F_{(2k+1)n}(c)}{F_{2k+1}(c)} + \frac{F_{(2k+1)(n+1)}(c)}{F_{2k+1}(c)}\right), & \text{if $r=k$},\\[6pt]
			(-1)^{\binom{k+1}{2}}\big(L_k(c)+L_{k+1}(c)\big)\frac{F_{(2k+1)(n+1)}(c)}{F_{2k+1}(c)}, & \text{if $r=k+1$},\\[6pt]
			(-1)^{k}\big(L_k(c)+L_{k+1}(c)\big)\frac{F_{(2k+1)(n+1)}(c)}{F_{2k+1}(c)}, & \text{if $r=2k$},\\[6pt]
			0, & \text{otherwise}.
		\end{cases}
	\end{align}
	In addition, the $r=0$ and $k$ cases also hold when $k=0$. Namely,
	\begin{align}
		\det_{0\leq i,j \leq n-1}\big(\Cat_{i+j+1,0}^{(1,c)}\big) = F_n(c) + F_{n+1}(c).
	\end{align}
\end{theorem}

\begin{theorem}[Cigler~\cite{Cig2023}, Conjecture~9]\label{th:b1-m2}
	Let $k\ge 2$ and write $N=(2k+1)n+r$ with $0\le r\le 2k$. Then for $c$ an indeterminate,
	\begin{align}\label{eq:b1-m2}
		&\!\!\!\!\det_{0\leq i,j \leq N-1}\big(\Cat_{i+j+2,k}^{(1,c)}\big)\notag\\
		&\quad = \begin{cases}
			\left(\frac{F_{(2k+1)n}(c)}{F_{2k+1}(c)} + \frac{F_{(2k+1)(n+1)}(c)}{F_{2k+1}(c)}\right)^2, & \text{if $r=0$},\\[6pt]
			(-1)^{\binom{k-1}{2}}\left(\frac{F_{(2k+1)n}(c)}{F_{2k+1}(c)} + \frac{F_{(2k+1)(n+1)}(c)}{F_{2k+1}(c)}\right)^2, & \text{if $r=k-1$},\\[6pt]
			(-1)^{\binom{k+1}{2}}\big(L_{k+1}(c)+L_{k}(c)\big)^2\left(\frac{F_{(2k+1)(n+1)}(c)}{F_{2k+1}(c)}\right)^2, & \text{if $r=k+1$},\\[6pt]
			-\big(L_{k+1}(c)+L_{k}(c)\big)^2\left(\frac{F_{(2k+1)(n+1)}(c)}{F_{2k+1}(c)}\right)^2, & \text{if $r=2k-1$},\\[6pt]
			0, & \text{if $r\ne \begin{aligned}
					&0,k,k\pm 1,\\
					&2k, 2k- 1,
				\end{aligned}$}
		\end{cases}
	\end{align}
	and
	\begin{align}\label{eq:b1-m2-sum}
		&\det_{0\leq i,j \leq (2k+1)n+2k-1}\big(\Cat_{i+j+2,k}^{(1,c)}\big) + (-1)^{\binom{k-1}{2}}\det_{0\leq i,j \leq (2k+1)n+k-1}\big(\Cat_{i+j+2,k}^{(1,c)}\big)\notag\\
		&\quad\quad\quad\quad\quad\quad = (-1)^k \big(L_{k+1}(c)+L_{k}(c)\big)^2 \big(F'_{k+1}(c)-F'_k(c)\big)\left(\frac{F_{(2k+1)(n+1)}(c)}{F_{2k+1}(c)}\right)^2,
	\end{align}
	where $F'(c) := \frac{\ddd}{\ddd x}\big\vert_{x\mapsto c} F(x)$. In addition, the $r=0$ case in \eqref{eq:b1-m2} and the relation \eqref{eq:b1-m2-sum} hold when $k=1$. Namely,
	\begin{align}
		\det_{0\leq i,j \leq 3n-1}\big(\Cat_{i+j+2,1}^{(1,c)}\big) = \left(\frac{F_{3n}(c)}{c^2-1} + \frac{F_{3n+3}(c)}{c^2-1}\right)^2,
	\end{align}
	and
	\begin{align}
		\det_{0\leq i,j \leq 3n+1}\big(\Cat_{i+j+2,1}^{(1,c)}\big) + \det_{0\leq i,j \leq 3n}\big(\Cat_{i+j+2,1}^{(1,c)}\big) = -\left(\frac{c+2}{c+1} F_{3n+3}(c)\right)^2.
	\end{align}
\end{theorem}

For generic $m$, we also have a result in the same manner as Theorem~\ref{th:b0-m}, but this time covering four residue classes.

\begin{theorem}\label{th:b1-m}
	Let $m\ge 1$ and write $N=(2k+1)n+r$ with $0\le r\le 2k$. Then for $c$ an indeterminate, we have, when $k\ge m-1$,
	\begin{align}
		&\det_{0\leq i,j \leq N-1}\big(\Cat_{i+j+m,k}^{(1,c)}\big)\notag\\
		&\qquad = \begin{cases}
			\left(\frac{F_{(2k+1)n}(c)}{F_{2k+1}(c)} + \frac{F_{(2k+1)(n+1)}(c)}{F_{2k+1}(c)}\right)^m, & \text{if $r=0$},\\[6pt]
			(-1)^{\binom{k-m+1}{2}} \left(\frac{F_{(2k+1)n}(c)}{F_{2k+1}(c)} + \frac{F_{(2k+1)(n+1)}(c)}{F_{2k+1}(c)}\right)^m, & \text{if $r=k-m+1$};
		\end{cases}
	\end{align}
	and when $k\ge m$,
	\begin{align}
		&\det_{0\leq i,j \leq N-1}\big(\Cat_{i+j+m,k}^{(1,c)}\big)\notag\\
		&\qquad = \begin{cases}
			(-1)^{\binom{k+1}{2}}\big(L_{k+1}(c)+L_{k}(c)\big)^m\left(\frac{F_{(2k+1)(n+1)}(c)}{F_{2k+1}(c)}\right)^m, & \text{if $r=k+1$},\\[6pt]
			(-1)^{\binom{m}{2}+mk}\big(L_{k+1}(c)+L_{k}(c)\big)^m\left(\frac{F_{(2k+1)(n+1)}(c)}{F_{2k+1}(c)}\right)^m, & \text{if $r=2k-m+1$}.
		\end{cases}
	\end{align}
\end{theorem}

\begin{remark}
	Several cases of \cite[Conjecture~9]{Cig2023} were stated in a different form, and they can be recovered by the relation \eqref{eq:F-sum=L-diff}. These changes are made to maintain the consistency between Theorem~\ref{th:b1-m2} and Theorem~\ref{th:b1-m}.
\end{remark}

As an application, we establish a binomial determinant evaluation conjectured by Cigler and Krattenthaler~\cite[p.~171, Conjecture~24]{CK2011}.

\begin{theorem}[Cigler--Krattenthaler~\cite{CK2011}, Conjecture~24]\label{th:CK}
	Let $k\ge 2$ and write $N=(2k-1)n+r$ with $0\le r\le 2k-2$. Then,
	\begin{align}\label{eq:CK-k}
		\det_{0\leq i,j \leq N-1} \binom{2i+2j+3}{i+j+k+1} = \begin{cases}
			2n+1, & \text{if $r=0$},\\
			(-1)^{\binom{k-1}{2}}(2n+1), & \text{if $r=k-1$},\\
			(-1)^{\binom{k}{2}}(4n+4), & \text{if $r=k$},\\
			(-1)^{k-1}(4n+4), & \text{if $r=2k-2$},\\
			0, & \text{otherwise}.
		\end{cases}
	\end{align}
	In addition, the $r=0$ case also holds when $k=1$. Namely,
	\begin{align}\label{eq:CK-1}
		\det_{0\leq i,j \leq n-1} \binom{2i+2j+3}{i+j+2} = 2n+1.
	\end{align}
\end{theorem}

\begin{remark}
	The evaluation \eqref{eq:CK-1} was proven by Krattenthaler and Yaqubi in \cite[p.~263, Corollary~18]{KY2018}.
\end{remark}

\textbf{Outline of the paper.} In Section~\ref{sec:comb}, we interpret the (shifted) Hankel determinants of Catalan-like sequences as the weighted enumerations of nonintersecting Motzkin meanders, and explain the two results of Cigler and Krattenthaler in Theorems~\ref{th:b0-m0} and \ref{th:b0-m1} in this combinatorial context. In Section~\ref{sec:orthogonal}, we provide some preliminary techniques from the theory of orthogonal polynomials for Hankel determinant evaluations. In Section~\ref{sec:F-L}, we collect some basic properties of the polynomials $F_n(x)$ and $L_n(x)$. With these preparations, we prove our Hankel determinant expressions for $b=0$ and $1$ in Sections~\ref{sec:b0} and \ref{sec:b1}, respectively. Finally, we confirm the binomial determinant conjecture of Cigler and Krattenthaler in Section~\ref{sec:binomial}.

\section{Motzkin meanders, a combinatorial perspective}\label{sec:comb}

A \emph{Motzkin meander} $\cM$ is a lattice path starting from a certain point $(a_1,0)$ on the horizontal axis, ending at another point $(b_1,b_2)$ in the upper half plane, and never falling below the horizontal axis such that only \emph{up} steps $\rU:=(1,1)$, \emph{down} steps $\rD:=(1,-1)$, and \emph{level} steps $\rL:=(1,0)$ are used.\footnote{Strictly speaking, it is usually required that a Motzkin meander starts from the origin $(0,0)$. Here we have loosened this condition for a more general consideration.} For every level step $\rL$, we define its \emph{height} $h(\rL)$ as the vertical coordinate of this step.

Let $\bs=(s_k)_{k\geq0}$ be a fixed sequence. For every Motzkin meander $\cM$, we assign a \emph{weight}
\begin{align*}
	w^{(\bs)}(\cM) := \prod_{\rL\in \cM} s_{h(\rL)},
\end{align*}
where the product runs over all level steps in $\cM$. Summing over all Motzkin meanders from a certain point $(a_1,0)$ to another point $(b_1,b_2)$ such that $b_1\ge a_1$, we have
\begin{align}\label{eq:M-gf}
	\sum_{\cM:(a_1,0)\to (b_1,b_2)} w^{(\bs)}(\cM) = \Cat_{b_1-a_1,b_2}^{(\bs)},
\end{align}
by recalling the recurrence \eqref{eq:a-rec}.

Now we say a family of Motzkin meanders is \emph{nonintersecting} if no two paths in the family intersect at a lattice point. Assign $N$ starting points
\begin{align*}
	A_i := (-i,0),\qquad 0\le i\le N-1,
\end{align*}
and $N$ ending points
\begin{align*}
	B_j := (j+m,k),\qquad 0\le j\le N-1.
\end{align*}
The following result follows immediately from the well-known \emph{Lindstr\"om--Gessel--Viennot lemma}~\cite{Lin1973,GV1985} with \eqref{eq:M-gf} utilized.

\begin{proposition}\label{prop:M-nonintersecting}
	Let $k\ge 0$. Then,
	\begin{align}
		\sum_{\sigma\in \mathfrak{S}_N}\sgn(\sigma)
		\sum_{\substack{\dcP:=(\cM_0,\ldots,\cM_{N-1})\\
				\cM_i:A_i\to B_{\sigma(i)}\\
				\text{$\dcP$ nonintersecting}}} \prod_{i=0}^{N-1} w^{(\bs)}(\cM_i) = \det_{0\leq i,j \leq N-1}\big(\Cat_{i+j+m,k}^{(\bs)}\big),
	\end{align}
	where $\mathfrak{S}_N$ is the symmetric group and $\sgn(\sigma)$ is the signature of $\sigma\in \mathfrak{S}_N$.
\end{proposition}

In \cite{CK2011}, Cigler and Krattenthaler considered another weight for Motzkin meanders, given by
\begin{align*}
	w_{\mathrm{CK}}(\cM) := (x+y)^{\#_{\cM}(\rL)} (xy)^{\#_{\cM}(\rD)},
\end{align*}
where $\#_{\cM}(\rL)$ and $\#_{\cM}(\rD)$ denote the numbers of level and down steps in $\cM$, respectively. Now if we make the specialization
\begin{align*}
	y = x^{-1},
\end{align*}
then $w_{\mathrm{CK}}$ reduces to our weight $w^{(\bs)}$ with
\begin{align*}
	\bs = (c,c,c,\ldots),
\end{align*}
where
\begin{align*}
	c = x+x^{-1}.
\end{align*}

Write $N=(k+1)n+r$ with $0\le r\le k$. By virtue of Proposition~\ref{prop:M-nonintersecting}, it follows from \cite[p.~146, Theorem~1]{CK2011} that
\begin{align*}
	\det_{0\leq i,j \leq N-1}\big(\Cat_{i+j,k}^{(0,c)}\big) = \begin{cases}
		(-1)^{\binom{k+1}{2}n}, & \text{if $r=0$},\\
		0, & \text{otherwise},
	\end{cases}
\end{align*}
thereby giving Theorem~\ref{th:b0-m0}. In a similar vein, we know from \cite[p.~146, Theorem~2]{CK2011} that
\begin{align*}
	\det_{0\leq i,j \leq N-1}\big(\Cat_{i+j+1,k}^{(0,c)}\big) = \begin{cases}
		(-1)^{\binom{k+1}{2}n}\, \tfrac{x^{(k+1)(n+1)}-x^{-(k+1)(n+1)}}{x^{k+1}-x^{-(k+1)}}, & \text{if $r=0$},\\[6pt]
		(-1)^{\binom{k+1}{2}n + \binom{k}{2}}\, \tfrac{x^{(k+1)(n+1)}-x^{-(k+1)(n+1)}}{x^{k+1}-x^{-(k+1)}}, & \text{if $r=k$},\\[6pt]
		0, & \text{otherwise}.
	\end{cases}
\end{align*}
Finally, according to the Binet-type formula \eqref{eq:Binet-F} for the polynomials $F_n$,
\begin{align*}
	F_n(c) = F_n(x+x^{-1}) = \frac{x^n - x^{-n}}{x - x^{-1}}.
\end{align*}
Thus, Theorem~\ref{th:b0-m1} is also true.

\section{Hankel determinants and orthogonal polynomials}\label{sec:orthogonal}

The evaluation of Hankel determinants usually builds on techniques from the theory of orthogonal polynomials. We say a family of \emph{monic} polynomials $(p_n(x))_{n\ge 0}$ with $p_n(x)$ of degree $n$ is \emph{orthogonal} if there is a linear functional $L$ on $\mathbb{C}[x]$ such that $L\big(p_m(x)p_n(x)\big) = \delta_{m,n} \kappa_n$ where $\delta_{m,n}$ is the Kronecker delta and $(\kappa_n)_{n\ge 0}$ is a fixed \emph{nonzero} sequence.

By a standard result of Stieltjes~\cite{Sti1894}, also more commonly referred to as the \emph{Favard theorem}~\cite{Fav1935}, we know that polynomials $(p_n(x))_{n\ge 0}$ are orthogonal if and only if they satisfy a three-term recursive relation. Namely, there exist sequences $(a_n)_{n\ge 0}$ and $(b_n)_{n\ge 1}$ with $b_n\ne 0$ such that $p_0(x)=1$, $p_1(x)=x+a_0$, and for $n\ge 1$,
\begin{align}\label{eq:3TermRec}
	p_{n+1}(x) = (x + a_n)p_n(x) - b_n p_{n-1}(x).
\end{align}
See \cite[p.~21, Theorem~12]{Kra1999}. Now we may combine the Jacobi continued fraction expression for the generating function of the moments $L(x^n)$ (see \cite[p.~20, Theorem~11]{Kra1999}) and a classic result of Heilermann~\cite{Hei1846} (see also \cite[p.~21, Theorem~13]{Kra1999}) to get the following result.

\begin{lemma}\label{le:det-general}
	Let the monic polynomials $(p_n(x))_{n\ge 0}$ be orthogonal associated with the linear functional $L$. Then,
	\begin{align}\label{eq:det-general}
		\underset{{0\le i,j\le N-1}}{\det}\big(L(x^{i+j})\big) = L(x^0)^{N} b_1^{N-1} b_2^{N-2} \cdots b_{N-2}^2 b_{N-1},
	\end{align}
	where the $b$'s are as in \eqref{eq:3TermRec}.
\end{lemma}

To move on from the Hankel determinants for the moments $L(x^n)$ to the shifted ones, or to even more general cases, we need a result first discovered by Br\'ezin and Hikami~\cite{BH2000}, the history of which was surveyed by Krattenthaler~\cite{Kra2023}.

Let $x_1,\ldots,x_d$ be indeterminates. Then \cite[p.~598, Theorem~1]{Kra2023} asserts that
\begin{align}
	\frac{\underset{{0\le i,j\le N-1}}{\det}\big(L(x^{i+j}\prod_{l=1}^d (x-x_i))\big)}{\underset{{0\le i,j\le N-1}}{\det}\big(L(x^{i+j})\big)} = (-1)^{dN}\, \frac{\underset{{1\le i,j\le d}}{\det} \big(p_{N+i-1}(x_j)\big)}{\prod\limits_{1\le i<j\le d}(x_j-x_i)}.
\end{align}
Although it seems that the right-hand side of the above relation has a zero denominator if two of the $x_i$'s are identical, Krattenthaler~\cite[p.~611]{Kra2023} argued that it is indeed a \emph{polynomial} in $x_1,\ldots,x_d$ because the denominator divides the numerator. Furthermore, with repeated $x_i$'s, this ratio can be evaluated by \cite[p.~611, Proposition~5]{Kra2023}. In our context, we need the following specialization.

\begin{lemma}[\cite{Kra2023}, Theorem~1 + Proposition~5]\label{le:Kra}
	Let the monic polynomials $(p_n(x))_{n\ge 0}$ be orthogonal associated with the linear functional $L$. Let $x_1,\ldots,x_k$ be nonzero and pairwise distinct and put
	\begin{align*}
		u(x) := (x-x_1)\cdots (x-x_k).
	\end{align*}
	Then,
	\begin{align}\label{eq:Kra}
		\frac{\underset{{0\le i,j\le N-1}}{\det}\big(L(x^{i+j+m}u(x))\big)}{\underset{{0\le i,j\le N-1}}{\det}\big(L(x^{i+j})\big)} = (-1)^{(k+m)N}\, \frac{\det(P,P')}{u(0)^m\prod\limits_{1\le i<j\le k}(x_j-x_i)},
	\end{align}
	where the block matrix $(P,P')$ is $(k+m)\times (k+m)$, given by
	\begin{align*}
		P := \big(p_{N+i}(x_j)\big)_{\substack{0\le i\le k+m-1\\1\le j\le k}},\qquad P' := \big(\tfrac{\ddd^l}{\ddd x^l}\big\vert_{x\mapsto 0} \tfrac{p_{N+i}(x)}{l!}\big)_{\substack{0\le i\le k+m-1\\0\le l\le m-1}}.
	\end{align*}
\end{lemma}

\begin{proof}
	In \cite[p.~611, Proposition~5]{Kra2023}, we put $y_i = -x_i$ and $m_i = 1$ for $1\le i\le k$, followed by $y_{k+1} = 0$ and $m_{k+1} = m$.
\end{proof}

We record two corollaries to simplify the analysis in later sections. In what follows, we write $[a,b] := \{n\in \mathbb{Z}: a\le n\le b\}$.

\begin{corollary}\label{coro:Kra-1}
	Under the assumptions in Lemma~\ref{le:Kra}, suppose there is a subset $S \subset [0,k+m-1]$ of cardinality $m+1$ such that for every $s\in S$, we can find a certain $\tau_s\in [0,k+m-1]\backslash S$ and a constant $c_s \in \mathbb{C}$ with
	\begin{align*}
		p_{N+s}(x_j) = c_s p_{N+\tau_s}(x_j)
	\end{align*}
	for all $j\in [1,k]$. Then,
	\begin{align}
		\underset{{0\le i,j\le N-1}}{\det}\big(L(x^{i+j+m}u(x))\big) = 0.
	\end{align}
\end{corollary}

\begin{remark}
	By taking $c_s = 0$ and allowing $\tau_s$ to be arbitrary (even if the complement $[0,k+m-1]\backslash S$ is empty), the condition in Corollary~\ref{coro:Kra-1} covers the case where $p_{N+s}(x_j) = 0$ for all $j\in [1,k]$.
\end{remark}

\begin{proof}
	Note that in the assumed context, the matrix $P$ in Lemma~\ref{le:Kra} has rank at most $k-1$, and hence the block matrix $(P,P')$ has rank at most $k+m-1$. This implies that $\det(P,P')$ vanishes, and so does our desired determinant by invoking the master relation \eqref{eq:Kra}.
\end{proof}

\begin{corollary}\label{coro:Kra-2}
	Under the assumptions in Lemma~\ref{le:Kra}, suppose there is a permutation $\sigma\in \mathfrak{S}_{k+m}$ such that we can find mappings $f:[0,k-1]\to \mathbb{C}$ and $g:[1,k]\to \mathbb{C}$ with
	\begin{align*}
		p_{N+\sigma(i)}(x_j) = f(i)g(j) p_{i}(x_j)
	\end{align*}
	for all $i\in [0,k-1]$ and $j\in [1,k]$, and such that for every $s\in [k,k+m-1]$, we can find a certain $\tau_s\in [0,k-1]$ and a constant $c_s \in \mathbb{C}$ with
	\begin{align*}
		p_{N+\sigma(s)}(x_j) = c_s p_{N+\sigma(\tau_s)}(x_j)
	\end{align*}
	for all $j\in [1,k]$. Then,
	\begin{align}\label{eq:Kra-2}
		\frac{\underset{{0\le i,j\le N-1}}{\det}\big(L(x^{i+j+m}u(x))\big)}{\underset{{0\le i,j\le N-1}}{\det}\big(L(x^{i+j})\big)}& = \sgn(\sigma) (-1)^{(k+m)N} \frac{1}{u(0)^m} \prod_{i=0}^{k-1}f(i) \prod_{j=1}^k g(j)\notag\\
		&\times \underset{{\substack{k\le s\le k+m-1\\0\le l\le m-1}}}{\det} \left(\frac{\ddd^l}{\ddd x^l}\bigg\vert_{x\mapsto 0} \!\!\frac{p_{N+\sigma(s)}(x) - c_s p_{N+\sigma(\tau_s)}(x)}{l!}\right).
	\end{align}
\end{corollary}

\begin{proof}
	Interchanging the rows of the block matrix $(P,P')$ in Lemma~\ref{le:Kra}, we have
	\begin{align*}
		\det(P,P') = \sgn(\sigma) \det(P_\sigma,P'_\sigma),
	\end{align*}
	where
	\begin{align*}
		P_\sigma := \big(p_{N+\sigma(i)}(x_j)\big)_{\substack{0\le i\le k+m-1\\1\le j\le k}},\qquad P'_\sigma := \big(\tfrac{\ddd^l}{\ddd x^l}\big\vert_{x\mapsto 0} \tfrac{p_{N+\sigma(i)}(x)}{l!}\big)_{\substack{0\le i\le k+m-1\\0\le l\le m-1}}.
	\end{align*}
	Now for every $s\in [k,k+m-1]$, we subtract $c_s$ times the $\tau_s$-th row of $(P_\sigma,P'_\sigma)$ from the $s$-th row. Then we arrive at a block matrix
	\begin{align*}
		\begin{pmatrix}
			A & B\\
			0 & D
		\end{pmatrix},
	\end{align*}
	where in particular, $A$ is a $k\times k$ square matrix
	\begin{align*}
		A = \big(f(i)g(j)p_i(x_j)\big)_{\substack{0\le i\le k-1\\1\le j\le k}},
	\end{align*}
	and $D$ is an $m\times m$ square matrix
	\begin{align*}
		D = \left(\frac{\ddd^l}{\ddd x^l}\bigg\vert_{x\mapsto 0} \!\!\frac{p_{N+\sigma(s)}(x) - c_s p_{N+\sigma(\tau_s)}(x)}{l!}\right)_{\substack{k\le s\le k+m-1\\0\le l\le m-1}}.
	\end{align*}
	Moreover,
	\begin{align*}
		\det(P,P') = \sgn(\sigma) \det(P_\sigma,P'_\sigma) = \sgn(\sigma) \det\begin{pmatrix}
			A & B\\
			0 & D
		\end{pmatrix} = \sgn(\sigma) \det A \det D.
	\end{align*}
	Note that $\det D$ is exactly the determinant on the right-hand side of \eqref{eq:Kra-2}. For $\det A$, we have
	\begin{align*}
		\det A = \prod_{i=0}^{k-1}f(i) \prod_{j=1}^k g(j) \det_{\substack{0\le i\le k-1\\1\le j\le k}} \big(p_i(x_j)\big).
	\end{align*}
	Since $p_i(x)$ is monic and of degree $i$ for every $i$, we write
	\begin{align*}
		p_i(x) := p_{i,0} + p_{i,1}x + \cdots + p_{i,i-1}x^{i-1} + x^i.
	\end{align*}
	Thus,
	\begin{align*}
		&\begin{pmatrix}
			p_0(x_1) & p_0(x_2) & \cdots & p_0(x_k)\\
			p_1(x_1) & p_1(x_2) & \cdots & p_1(x_k)\\
			\vdots & \vdots & \ddots & \vdots\\
			p_{k-1}(x_1) & p_{k-1}(x_2) & \cdots & p_{k-1}(x_k)\\
		\end{pmatrix}\\
		&\qquad\qquad =
		\begin{pmatrix}
		1 & 0 & \cdots & 0\\
		p_{1,0} & 1 & \cdots & 0\\
		\vdots & \vdots & \ddots & \vdots\\
		p_{k-1,0} & p_{k-1,1} & \cdots & 1\\
		\end{pmatrix}
		\begin{pmatrix}
			1 & 1 & \cdots & 1\\
			x_1 & x_2 & \cdots & x_k\\
			\vdots & \vdots & \ddots & \vdots\\
			x_1^{k-1} & x_2^{k-1} & \cdots & x_k^{k-1}\\
		\end{pmatrix}.
	\end{align*}
	The determinant of the lower triangular matrix in the above is $1$, and the determinant of the Vandermonde matrix, according to the standard evaluation \cite[p.~5, eq.~(2.1)]{Kra1999}, equals $\prod_{1\le i<j\le k}(x_j-x_i)$. Hence,
	\begin{align*}
		\det A = \prod_{1\le i<j\le k}(x_j-x_i).
	\end{align*}
	Finally, our claimed result follows by recalling \eqref{eq:Kra}.
\end{proof}

\section{The polynomials $F_n(x)$ and $L_n(x)$}\label{sec:F-L}

Before proceeding with the evaluation of our Hankel determinants, we collect some basic properties of the polynomials $F_n(x)$ and $L_n(x)$ for later use.

First, the following \emph{Binet-type formulas} are well-known.

\begin{lemma}
	We have
	\begin{align}\label{eq:Binet-F}
		F_n(z+z^{-1}) = z^{-n+1} + z^{-n+3} + \cdots + z^{n-3} + z^{n-1} = \frac{z^n - z^{-n}}{z - z^{-1}},
	\end{align}
	and
	\begin{align}\label{eq:Binet-L}
		L_n(z+z^{-1}) = z^n + z^{-n}.
	\end{align}
\end{lemma}

The two Binet-type formulas have several implications that play an important role in our analysis. Here we always make the parametrization $x:=z+z^{-1}$.

\begin{proposition}
	We have
	\begin{align}
		F_n(-x) &= (-1)^{n-1} F(x),\label{eq:F(-x)}\\
		L_n(-x) &= (-1)^{n} L(x).\label{eq:L(-x)}
	\end{align}
\end{proposition}

\begin{proof}
	We replace $z$ with $-z$ in \eqref{eq:Binet-F} and \eqref{eq:Binet-L}.
\end{proof}

\begin{proposition}
	We have, for $0\le i\le n$,
	\begin{align}\label{eq:F_n+-i}
		F_{n+i}(x) + F_{n-i}(x) = F_n(x) L_i(x).
	\end{align}
\end{proposition}

\begin{proof}
	We have
	\begin{align*}
		F_{n+i}(x) + F_{n-i}(x) &= \frac{z^{n+i} - z^{-n-i}}{z - z^{-1}} + \frac{z^{n-i} - z^{-n+i}}{z - z^{-1}}\\
		& = \frac{z^n - z^{-n}}{z - z^{-1}} (z^i + z^{-i})\\
		& = F_n(x) L_i(x),
	\end{align*}
	as claimed.
\end{proof}

\begin{proposition}
	We have
	\begin{align}\label{eq:F-sum=L-diff}
		(x-2)\big(F_{n+1}(x) + F_n(x)\big) = L_{n+1}(x) - L_n(x).
	\end{align}
\end{proposition}

\begin{proof}
	We have
	\begin{align*}
		(x-2)\big(F_{n+1}(x) + F_n(x)\big) &= (z+z^{-1}-2)\left(\frac{z^{n+1} - z^{-n-1}}{z - z^{-1}}+\frac{z^n - z^{-n}}{z - z^{-1}}\right)\\
		&= (z^{n+1} + z^{-n-1}) - (z^{n} + z^{-n})\\
		&= L_{n+1}(x) - L_n(x),
	\end{align*}
	as claimed.
\end{proof}

\begin{proposition}
	We have
	\begin{align}\label{eq:F-L}
		F_{st}(x) = F_s(x) F_t\big(L_s(x)\big).
	\end{align}
\end{proposition}

\begin{proof}
	We have
	\begin{align*}
		F_s(x) F_t\big(L_s(x)\big) = F_s(z+z^{-1}) F_t(z^s+z^{-s}) = \frac{z^s - z^{-s}}{z - z^{-1}}\frac{z^{st} - z^{-st}}{z^s - z^{-s}} = F_{st}(x),
	\end{align*}
	as claimed.
\end{proof}

\begin{proposition}
	Define
	\begin{align}\label{eq:D-def}
		D_n(x) := F_{n+1}(x) - F_n(x).
	\end{align}
	We have
	\begin{align}\label{eq:D-1}
		D_{(2s+1)t+s}(x) = D_s(x) D_{t}\big(L_{2s+1}(x)\big).
	\end{align}
	Also, for $0\le i\le (2s+1)t+2s$,
	\begin{align}\label{eq:D-3}
		&D_{(2s+1)t+2s+1+i}(x) - D_{(2s+1)t+2s-i}(x)\notag\\
		&\qquad = D_s(x) \big(L_{s+1}(x) - L_s(x)\big) F_{t+1}\big(L_{2s+1}(x)\big) \big(F_{i+1}(x)+F_i(x)\big).
	\end{align}
\end{proposition}

\begin{proof}
	We have
	\begin{align*}
		D_{(2s+1)t+s}(x) = \frac{1+z^{(1+2s)(1+2t)}}{z^{2st+s+t}(1+z)} = D_s(x)\left(\frac{F_{(2s+1)(t+1)}(x)}{F_{2s+1}(x)} - \frac{F_{(2s+1)t}(x)}{F_{2s+1}(x)}\right).
	\end{align*}
	Therefore, \eqref{eq:D-1} follows by applying \eqref{eq:F-L}. For \eqref{eq:D-3}, we have
	\begin{align*}
		D_{(2s+1)t+2s+1+i}(x) &- D_{(2s+1)t+2s-i}(x)\\
		& = \frac{(1-z^{2(1+2s)(1+t)})(1-z^{1+2i})}{z^{(1+2s)(1+t)+i}(1+z)}\\
		& = D_s(x) \big(L_{s+1}(x) - L_s(x)\big) F_{t+1}\big(L_{2s+1}(x)\big) \big(F_{i+1}(x)+F_i(x)\big),
	\end{align*}
	where we have applied \eqref{eq:F-L} to $F_{t+1}\big(L_{2s+1}(x)\big)$ in the last equality.
\end{proof}

\begin{proposition}
	We have
	\begin{align}
		\frac{\ddd}{\ddd x}\bigg\vert_{x\mapsto -x}F_n(x) &= (-1)^{n} \frac{\ddd}{\ddd x}F_n(x),\label{eq:F'(-x)}\\
		\frac{\ddd}{\ddd x}\bigg\vert_{x\mapsto -x}L_n(x) &= (-1)^{n-1} \frac{\ddd}{\ddd x}L_n(x).\label{eq:L'(-x)}
	\end{align}
\end{proposition}

\begin{proof}
	By \eqref{eq:F(-x)}, we know that the polynomial $F_n(x)$ is an odd function if $n$ is even, and an even function if $n$ is odd. Since taking the derivative switches the parity, the relation \eqref{eq:F'(-x)} immediately follows. Finally, a similar argument justifies \eqref{eq:L'(-x)}.
\end{proof}

\begin{proposition}
	We have
	\begin{align}\label{eq:L'}
		\frac{\ddd}{\ddd x} L_n(x) = n F_n(x).
	\end{align}
\end{proposition}

\begin{proof}
	Recall that $L_n(z+z^{-1}) = z^n + z^{-n}$. By the chain rule,
	\begin{align*}
		\frac{\ddd}{\ddd x}\bigg\vert_{x\mapsto z+z^{-1}} L_n(x) = \frac{\frac{\ddd}{\ddd z}(z^n + z^{-n})}{\frac{\ddd}{\ddd z}(z+z^{-1})} = n \frac{z^n - z^{-n}}{z - z^{-1}} = n F(x),
	\end{align*}
	as claimed.
\end{proof}

\begin{proposition}
	We have
	\begin{align}\label{eq:FF'}
		F_n(x) \frac{\ddd}{\ddd x} F_{n+1}(x) - F_{n+1}(x) \frac{\ddd}{\ddd x} F_{n}(x) = \sum_{j=1}^n F_j(x)^2.
	\end{align}
\end{proposition}

\begin{proof}
	It is equivalent to show
	\begin{align*}
		\frac{\ddd}{\ddd x} \frac{F_{n+1}(x)}{F_n(x)} = \sum_{j=1}^n \left(\frac{F_j(x)}{F_n(x)}\right)^2.
	\end{align*}
	By the chain rule,
	\begin{align*}
		\frac{\ddd}{\ddd x}\bigg\vert_{x\mapsto z+z^{-1}} \frac{F_{n+1}(x)}{F_n(x)} &= \frac{\frac{\ddd}{\ddd z}\frac{z^{n+1} - z^{-n-1}}{z^n - z^{-n}}}{\frac{\ddd}{\ddd z}(z+z^{-1})} = \frac{(z^{2n+1}-z^{-2n-1}) - (2n+1)(z-z^{-1})}{(z-z^{-1})(z^n - z^{-n})^2}\\
		&= \sum_{j=1}^n \left(\frac{z^j - z^{-j}}{z^n - z^{-n}}\right)^2 = \sum_{j=1}^n \left(\frac{F_j(x)}{F_n(x)}\right)^2,
	\end{align*}
	as claimed.
\end{proof}

\begin{proposition}\label{prop:det-Fst}
	We have
	\begin{align}
		\det\begin{pmatrix}
			F_{st}(x) & \frac{\ddd}{\ddd x} F_{st}(x)\\
			F_{s(t+1)}(x) & \frac{\ddd}{\ddd x} F_{s(t+1)}(x)
		\end{pmatrix} = s F_s(x) \sum_{j=1}^t F_{sj}(x)^2.
	\end{align}
\end{proposition}

\begin{proof}
	For convenience, we write $F'_n(x_0) := \frac{\ddd}{\ddd x}\big\vert_{x\mapsto x_0} F_n(x)$. In light of \eqref{eq:F-L},
	\begin{align*}
		&\det\begin{pmatrix}
			F_{st}(x) & F'_{st}(x)\\
			F_{s(t+1)}(x) & F'_{s(t+1)}(x)
		\end{pmatrix}\\
		&\qquad = \det\begin{pmatrix}
		F_s(x) F_t\big(L_s(x)\big) & F'_s(x) F_t\big(L_s(x)\big) + F_s(x) L'_s(x) F'_t\big(L_s(x)\big)\\
		F_s(x) F_{t+1}\big(L_s(x)\big) & F'_s(x) F_{t+1}\big(L_s(x)\big) + F_s(x) L'_s(x) F'_{t+1}\big(L_s(x)\big)
		\end{pmatrix}\\
		&\qquad = F_s(x)^2 L'_s(x)  \big[F_t\big(L_s(x)\big)F'_{t+1}\big(L_s(x)\big) - F_{t+1}\big(L_s(x)\big)F'_{t}\big(L_s(x)\big)\big]\\
		&\qquad = s F_s(x)^3 \sum_{j=1}^t F_j\big(L_s(x)\big)^2,
	\end{align*}
	where we have used \eqref{eq:L'} and \eqref{eq:FF'} in the last equality. Finally, applying \eqref{eq:F-L} again, the claimed identity follows.
\end{proof}

Next, we require a technical lemma, which will be frequently used in our determinant evaluations.

\begin{lemma}\label{le:det-derivative}
	Assume $(p_n(x))_{n\ge 0}$ is a family of monic polynomials with $p_n(x)$ of degree $n$. For any function $w(x)$ smooth at zero,
	\begin{align}
		\det_{0\le i,j\le m} \left(\frac{\ddd^j}{\ddd x^j}\bigg\vert_{x\mapsto 0}\frac{w(x)p_i(x)}{j!}\right) = w(0)^{m+1}.
	\end{align}
\end{lemma}

\begin{proof}
	It is clear that $(p_n(x))_{n\ge 0}$ form a basis of the vector space of polynomials in $x$. By certain row manipulations,
	\begin{align*}
		\det_{0\le i,j\le m} \left(\frac{\ddd^j}{\ddd x^j}\bigg\vert_{x\mapsto 0}\frac{w(x)p_i(x)}{j!}\right) = \det_{0\le i,j\le m} \left(\frac{\ddd^j}{\ddd x^j}\bigg\vert_{x\mapsto 0}\frac{w(x)x^i}{j!}\right).
	\end{align*}
	Note that
	\begin{align*}
		\frac{\ddd^j}{\ddd x^j}\bigg\vert_{x\mapsto 0}\frac{w(x)x^i}{j!} = \begin{cases}
			0, & j<i,\\
			w(0), & j=i.
		\end{cases}
	\end{align*}
	Thus, the matrix on the right-hand side of the above is upper triangular with every diagonal entry being $w(0)$. The claimed determinant therefore holds.
\end{proof}

Finally, we notice that $L_n(x)$ and $F_n(x)$ are essentially the \emph{Chebyshev polynomials} of the first and second kinds~\cite{MH2003}, $T_n(x)$ and $U_n(x)$, linked by the relations
\begin{align*}
	L_n(x) = 2T_n(\tfrac{x}{2}),\qquad\qquad
	F_n(x) = U_{n-1}(\tfrac{x}{2}),
\end{align*}
which can be seen from the recurrences satisfied by the two kinds of Chebyshev polynomials~\cite[pp.~ 2 and 4, eqs.~(1.3) and (1.6)]{MH2003}. In light of the definition of $U_n(x)$ given in \cite[p.~3, eq.~(1.4)]{MH2003}, we further have
\begin{align}\label{eq:F-eva}
	F_n(2\cos\theta) = \frac{\sin (n\theta)}{\sin \theta}.
\end{align}
Thus, the following relation is valid.

\begin{lemma}\label{le:F-factorization}
	For $n\ge 1$,
	\begin{align}
		F_n(x) = \prod_{j=1}^{n-1} (x - x_j),
	\end{align}
	where
	\begin{align*}
		x_j := 2\cos \tfrac{j \pi}{n},\qquad\qquad 1\le j\le n-1,
	\end{align*}
	are pairwise distinct.
\end{lemma}

\section{Case $b=0$}\label{sec:b0}

Define a family of polynomials
\begin{align}\label{eq:P-def}
	P_n(x) := F_{n+1}(x-c),\qquad\qquad n\ge 0.
\end{align}
In view of \eqref{eq:F-rec}, we know that $P_n(x)$ is monic and of degree $n$. Therefore, the ring of polynomials in $x$ is generated by $(P_n(x))_{n\ge 0}$. In particular, for an arbitrary polynomial $f(x)$ of degree $d$, there is a \emph{unique} expansion
\begin{align}
	f(x) = \sum_{l=0}^d f_l^\dagger P_l(x).
\end{align}
We then define a functional $\Phi$ on this polynomial ring by
\begin{align}
	\Phi(f) := f_0^\dagger.
\end{align}
It is clear that $\Phi$ is \emph{linear}.

Now we establish the following evaluation of the linear functional $\Phi$.

\begin{lemma}\label{le:Phi-eva}
	For $n,k\ge 0$,
	\begin{align}\label{eq:Phi-eva}
		\Phi\big(x^n P_k(x)\big) = \Cat_{n,k}^{(0,c)}.
	\end{align}
\end{lemma}

\begin{proof}
	By \eqref{eq:F-rec},
	\begin{align}\label{eq:P-rec}
		P_{k+1}(x) = (x-c) P_k(x) - P_{k-1}(x),
	\end{align}
	where we put $P_{-1}(x) := 0$. After rearranging the terms, we have
	\begin{align*}
		x P_k(x) = P_{k+1}(x) + s_k P_k(x) + P_{k-1}(x),
	\end{align*}
	where
	\begin{align*}
		(s_0,s_1,s_2,\ldots) := (c,c,c,\ldots).
	\end{align*}
	Multiplying the previous relation by $x^{n-1}$ and invoking the linearity of $\Phi$, it follows that
	\begin{align*}
		\Phi\big(x^n P_k(x)\big) = \Phi\big(x^{n-1} P_{k+1}(x)\big) + s_k \Phi\big(x^{n-1} P_k(x)\big) + \Phi\big(x^{n-1} P_{k-1}(x)\big).
	\end{align*}
	For $n = 0$, it is clear from the definition of $\Phi$ that $\Phi\big(x^n P_k(x)\big)$ equals $1$ if $k = 0$, and $0$ when $k\ge 1$. Thus, $\Phi\big(x^n P_k(x)\big)$ satisfies the same recurrence and initial values as $\Cat_{n,k}^{(0,c)}$, thereby showing the claimed equality.
\end{proof}

Next, we move on to the orthogonality relation.

\begin{lemma}\label{le:P-orthogonal}
	The monic polynomials $(P_n(x))_{n\ge 0}$ are orthogonal associated with the linear functional $\Phi$.
\end{lemma}

\begin{proof}
	We need to compute $\Phi\big(P_r(x)P_s(x)\big)$. Without loss of generality, assume $r\le s$. Write
	\begin{align*}
		P_r(x) := \sum_{j=0}^r P_{r,j} x^j,
	\end{align*}
	with $P_{r,r} = 1$ because $P_r(x)$ is monic and has degree $r$. By the linearity of $\Phi$,
	\begin{align*}
		\Phi\big(P_r(x)P_s(x)\big) = \sum_{j=0}^r P_{r,j} \Phi\big(x^j P_s(x)\big) = \sum_{j=0}^r P_{r,j} \Cat_{j,s}^{(0,c)},
	\end{align*}
	where we have applied \eqref{eq:Phi-eva}. According to \eqref{eq:a-rec}, it is true that $\Cat_{j,s}^{(0,c)} = 0$ whenever $j < s$, and that $\Cat_{s,s}^{(0,c)} = 1$. Thus, when $r<s$, we have the vanishing of $\Phi\big(P_r(x)P_s(x)\big)$, while when $r=s$, we have
	\begin{align*}
		\Phi\big(P_r(x)P_r(x)\big) = P_{r,r} \Cat_{r,r}^{(0,c)} = 1.
	\end{align*}
	Our claim is then valid.
\end{proof}

The above analysis leads us to the Hankel determinants for the moments $\Phi(x^n)$.

\begin{proposition}\label{prop:det-Phi}
	We have
	\begin{align}
		\underset{{0\le i,j\le N-1}}{\det}\big(\Phi(x^{i+j})\big) = 1.
	\end{align}
\end{proposition}

\begin{proof}
	In Lemma~\ref{le:det-general}, we use the three-term recursive relation \eqref{eq:P-rec} and recall from \eqref{eq:Phi-eva} that $\Phi(x^0) = \Cat_{0,0}^{(0,c)} = 1$.
\end{proof}

Finally, Lemma~\ref{le:F-factorization} tells us the following factorization of $P_n(x)$.

\begin{lemma}\label{le:P-factorization}
	For $n\ge 0$,
	\begin{align}\label{eq:P-factorization}
		P_n(x) = \prod_{j=1}^n (x-\rho_{n,j}),
	\end{align}
	where
	\begin{align*}
		\rho_{n,j} := c + 2\cos \tfrac{j \pi}{n+1},\qquad\qquad 1\le j\le n,
	\end{align*}
	are pairwise distinct.
\end{lemma}

For the moment, we fix $k\ge 0$ and define
\begin{align}
	\rho_{j} := c + 2\cos \tfrac{j \pi}{k+1},\qquad\qquad 1\le j\le k.
\end{align}
We have further properties of the polynomials $P(x)$ evaluated at $\rho_j$.

\begin{lemma}\label{le:P-property}
	Let $j$ with $1\le j\le k$ be arbitrary. For $r\ge 0$,
	\begin{align}\label{eq:P-1}
		P_{(k+1)n+r}(\rho_j) = (-1)^{nj} P_r(\rho_j).
	\end{align}
	Also,
	\begin{align}\label{eq:P-2}
		P_k(\rho_j) = P_{2k+1}(\rho_j) = 0,
	\end{align}
	and for $s$ with $1\le s\le k$,
	\begin{align}\label{eq:P-3}
		P_{k+s}(\rho_j) = -P_{k-s}(\rho_j).
	\end{align}
\end{lemma}

\begin{proof}
	Put $\theta_j := \frac{j \pi}{k+1}$. By \eqref{eq:F-eva}, we see that for any $N\ge 0$,
	\begin{align*}
		P_N(\rho_j) = \frac{\sin ((N+1)\theta_j)}{\sin \theta_j},
	\end{align*}
	which yields \eqref{eq:P-1}. For \eqref{eq:P-2}, it is immediate that $P_k(\rho_j) = 0$ by \eqref{eq:P-factorization}, which further gives $P_{2k+1}(\rho_j) = 0$ by invoking \eqref{eq:P-1} with $n=1$ and $r=k$. Finally,
	\begin{align*}
		P_{k+s}(\rho_j) &=\frac{\sin((k+s+1)\theta_j)}{\sin\theta_j} =\frac{\sin(j\pi+s\theta_j)}{\sin\theta_j}\\
		& =-\frac{\sin(j\pi-s\theta_j)}{\sin\theta_j} =-\frac{\sin((k-s+1)\theta_j)}{\sin\theta_j} = -P_{k-s}(\rho_j),
	\end{align*}
	thereby giving \eqref{eq:P-3}.
\end{proof}

Now we are ready to evaluate our Hankel determinants. By Lemma~\ref{le:Phi-eva},
\begin{align}
	\det_{0\leq i,j \leq N-1}\big(\Cat_{i+j+m,k}^{(0,c)}\big) = \det_{0\leq i,j \leq N-1}\big(\Phi(x^{i+j+m} P_k(x))\big).
\end{align} 
Hence, we take $p = P$ and $L = \Phi$ in Corollaries~\ref{coro:Kra-1} and \ref{coro:Kra-2}. Also, we choose
\begin{align*}
	u(x) = P_k(x) = (x-\rho_1)\cdots (x-\rho_k),
\end{align*}
with the latter equality coming from Lemma~\ref{le:P-factorization}. According to \eqref{eq:P-def},
\begin{align*}
	P_k(0) = F_{k+1}(-c) = (-1)^k F_{k+1}(c),
\end{align*}
where we have utilized \eqref{eq:F(-x)}. In particular, $P_k(0)$, as a function of $c$, is never identical to the zero function when $k\ge 0$. Finally, in Corollary~\ref{coro:Kra-2}, for the denominator of the left-hand side of \eqref{eq:Kra-2}, we recall from Proposition~\ref{prop:det-Phi} that
\begin{align*}
	\underset{{0\le i,j\le N-1}}{\det}\big(\Phi(x^{i+j})\big) = 1.
\end{align*}
In the rest of this section, we always write
\begin{align*}
	N=(k+1)n+r
\end{align*}
with $0\le r\le k$.

\subsection{Case $m=0$}

In this part, we reprove the result of Cigler and Krattenthaler given in Theorem~\ref{th:b0-m0}.

\subsubsection{Case $r\ne 0$}

Let $k\ge 1$ to make the choice of $r$ not vacuous. In the context of Corollary~\ref{coro:Kra-1}, we put $m=0$. Note that
\begin{align*}
	P_{N+(k-r)}(\rho_j) \overset{\eqref{eq:P-1}}{=} (-1)^{nj} P_k(\rho_j) \overset{\eqref{eq:P-2}}{=} 0.
\end{align*}
Since $1\le r\le k$, we have $k-r \in [0,k-1]$. The determinant in question then vanishes.

\subsubsection{Case $r= 0$}

Let $k\ge 0$. In the context of Corollary~\ref{coro:Kra-2}, we put $m=0$. Choose the permutation
\begin{align*}
	\sigma(a) = a,\qquad\qquad a\in [0,k-1],
\end{align*}
so that
\begin{align*}
	\sgn(\sigma) = 1.
\end{align*}
For $i\in [0,k-1]$, we have
\begin{align*}
	P_{N+\sigma(i)}(\rho_j) \overset{\eqref{eq:P-1}}{=} (-1)^{nj} P_{i}(\rho_j).
\end{align*}
Thus, by \eqref{eq:Kra-2},
\begin{align*}
	\det_{0\leq i,j \leq N-1}\big(\Phi(x^{i+j} P_k(x))\big) = \sgn(\sigma) (-1)^{kN} \prod_{j=1}^k (-1)^{nj} = (-1)^{\binom{k+1}{2}n},
\end{align*}
as desired.

\subsection{Case $m=1$}

In this part, we reprove the result of Cigler and Krattenthaler given in Theorem~\ref{th:b0-m1}.

\subsubsection{Case $r\ne 0,k$}

Let $k\ge 2$. In the context of Corollary~\ref{coro:Kra-1}, we put $m=1$. Note that
\begin{align*}
	P_{N+(k-r)}(\rho_j) \overset{\eqref{eq:P-1}}{=} (-1)^{nj} P_k(\rho_j) \overset{\eqref{eq:P-2}}{=} 0,
\end{align*}
and
\begin{align*}
	P_{N+(k-r+1)}(\rho_j) \overset{\eqref{eq:P-1}}{=} (-1)^{nj} P_{k+1}(\rho_j) \overset{\eqref{eq:P-3}}{=} -(-1)^{nj} P_{k-1}(\rho_j) \overset{\eqref{eq:P-1}}{=} -P_{N+(k-r-1)}(\rho_j).
\end{align*}
Since $1\le r\le k-1$, we have $S:=\{k-r,k-r+1\}\subset [0,k]$ and $\{k-r-1\}\subset [0,k]\backslash S$. The determinant in question then vanishes.

\subsubsection{Cases $r= 0$ and $k$}

These correspond to the $m=1$ case of Theorem~\ref{th:b0-m}, and we defer its proof to Section~\ref{sec:b0m}.

\subsection{Case $m=2$}

In this part, we prove Theorem~\ref{th:b0-m2}.

\subsubsection{Case $r\ne 0,k-1,k$}

Let $k\ge 3$. In the context of Corollary~\ref{coro:Kra-1}, we put $m=2$. Note that
\begin{align*}
	P_{N+(k-r)}(\rho_j) \overset{\eqref{eq:P-1}}{=} (-1)^{nj} P_k(\rho_j) \overset{\eqref{eq:P-2}}{=} 0,
\end{align*}
Also,
\begin{align*}
	P_{N+(k-r+1)}(\rho_j) \overset{\eqref{eq:P-1}}{=} (-1)^{nj} P_{k+1}(\rho_j) \overset{\eqref{eq:P-3}}{=} -(-1)^{nj} P_{k-1}(\rho_j) \overset{\eqref{eq:P-1}}{=} -P_{N+(k-r-1)}(\rho_j),
\end{align*}
and
\begin{align*}
	P_{N+(k-r+2)}(\rho_j) \overset{\eqref{eq:P-1}}{=} (-1)^{nj} P_{k+2}(\rho_j) \overset{\eqref{eq:P-3}}{=} -(-1)^{nj} P_{k-2}(\rho_j) \overset{\eqref{eq:P-1}}{=} -P_{N+(k-r-2)}(\rho_j).
\end{align*}
Since $1\le r\le k-2$, we have $S:=\{k-r,k-r+1,k-r+2\}\subset [0,k+1]$ and $\{k-r-1,k-r-2\}\subset [0,k+1]\backslash S$. The determinant in question then vanishes.

\subsubsection{Cases $r= 0$ and $k-1$}

These correspond to the $m=2$ case of Theorem~\ref{th:b0-m}, and we defer its proof to Section~\ref{sec:b0m}.

\subsubsection{Case $r= k$}

Let $k\ge 0$. In the context of Corollary~\ref{coro:Kra-2}, we put $m=2$. Choose the permutation
\begin{align*}
	\sigma(a) = \begin{cases}
		k-a, & \qquad\qquad a\in [0,k],\\
		k+1, & \qquad\qquad a=k+1,
	\end{cases}
\end{align*}
so that
\begin{align*}
	\sgn(\sigma) = (-1)^{\binom{k+1}{2}}.
\end{align*}
For $i\in [0,k-1]$, we have
\begin{align*}
	P_{N+\sigma(i)}(\rho_j) \overset{\eqref{eq:P-1}}{=} (-1)^{nj} P_{2k-i}(\rho_j) \overset{\eqref{eq:P-3}}{=} (-1)^{nj+1} P_{i}(\rho_j).
\end{align*}
In addition,
\begin{align*}
	P_{N+\sigma(k)}(\rho_j) \overset{\eqref{eq:P-1}}{=} (-1)^{nj} P_{k}(\rho_j) \overset{\eqref{eq:P-2}}{=} 0,
\end{align*}
and
\begin{align*}
	P_{N+\sigma(k+1)}(\rho_j) \overset{\eqref{eq:P-1}}{=} (-1)^{nj} P_{2k+1}(\rho_j) \overset{\eqref{eq:P-2}}{=} 0.
\end{align*}
Thus, by \eqref{eq:Kra-2},
\begin{align*}
	\det_{0\leq i,j \leq N-1}\big(\Phi(x^{i+j+2} P_k(x))\big)&= \sgn(\sigma) (-1)^{(k+2)N} \frac{1}{\big((-1)^kF_{k+1}(c)\big)^2} \prod_{j=1}^k (-1)^{nj+1}\\
	&\times \det\left.\begin{pmatrix}
		P_{(k+1)n+k}(x) & \frac{\ddd}{\ddd x}P_{(k+1)n+k}(x)\\[2pt]
		P_{(k+1)n+2k+1}(x) & \frac{\ddd}{\ddd x}P_{(k+1)n+2k+1}(x)
	\end{pmatrix}\right\vert_{x\mapsto 0}.
\end{align*}
Since $P_n(x) = F_{n+1}(x-c)$, we apply Proposition~\ref{prop:det-Fst} and find that the above determinant equals
\begin{align*}
	(k+1) F_{k+1}(-c) \sum_{j=1}^{n+1} F_{(k+1)j}(-c)^2 = (-1)^k (k+1) F_{k+1}(c) \sum_{j=1}^{n+1} F_{(k+1)j}(c)^2,
\end{align*}
where we have also used \eqref{eq:F(-x)}. We then arrive at the desired identity
\begin{align*}
	\det_{0\leq i,j \leq N-1}\big(\Phi(x^{i+j+2} P_k(x))\big) = (-1)^{\binom{k+1}{2}n + \binom{k}{2}}(k+1) F_{k+1}(c) \sum_{j=1}^{n+1} \left(\tfrac{F_{(k+1)j}(c)}{F_{k+1}(c)}\right)^2.
\end{align*}

\subsection{Generic $m$}\label{sec:b0m}

In this part, we prove Theorem~\ref{th:b0-m}. Throughout, fix $m\ge 1$.

\subsubsection{Case $r= 0$}\label{sec:b0m-a}

Let $k\ge m-1$. In the context of Corollary~\ref{coro:Kra-2}, we choose the permutation
\begin{align*}
	\sigma(a) = a,\qquad\qquad a\in [0,k+m-1],
\end{align*}
so that
\begin{align*}
	\sgn(\sigma) = 1.
\end{align*}
For $i\in [0,k-1]$, we have
\begin{align*}
	P_{N+\sigma(i)}(\rho_j) \overset{\eqref{eq:P-1}}{=} (-1)^{nj} P_{i}(\rho_j).
\end{align*}
In addition,
\begin{align*}
	P_{N+\sigma(k)}(\rho_j) \overset{\eqref{eq:P-1}}{=} (-1)^{nj} P_{k}(\rho_j) \overset{\eqref{eq:P-2}}{=} 0,
\end{align*}
and for $s\in [k+1, k+m-1]$,
\begin{align*}
	P_{N+\sigma(s)}(\rho_j) \overset{\eqref{eq:P-1}}{=} (-1)^{nj} P_{s}(\rho_j) \overset{\eqref{eq:P-3}}{=} -(-1)^{nj} P_{2k-s}(\rho_j) \overset{\eqref{eq:P-1}}{=} -P_{N+(2k-s)}(\rho_j).
\end{align*}
Thus, by \eqref{eq:Kra-2},
\begin{align*}
	\det_{0\leq i,j \leq N-1}\big(\Phi(x^{i+j+m} P_k(x))\big) = \sgn(\sigma) (-1)^{(k+m)N} \frac{\det M_P}{\big((-1)^kF_{k+1}(c)\big)^m} \prod_{j=1}^k (-1)^{nj},
\end{align*}
where
\begin{align}\label{eq:MP}
	M_P := \begin{pmatrix}
		\frac{\ddd^0}{\ddd x^0}\big\vert_{x\mapsto 0}\frac{P_{(k+1)n+k}(x)}{0!} & \cdots & \frac{\ddd^{m-1}}{\ddd x^{m-1}}\big\vert_{x\mapsto 0} \frac{P_{(k+1)n+k}(x)}{(m-1)!}\\[2pt]
		\frac{\ddd^0}{\ddd x^0}\big\vert_{x\mapsto 0}\frac{P_{(k+1)n+k\pm 1}(x)}{0!} & \cdots & \frac{\ddd^{m-1}}{\ddd x^{m-1}}\big\vert_{x\mapsto 0} \frac{P_{(k+1)n+k\pm 1}(x)}{(m-1)!}\\[2pt]
		\vdots & \ddots & \vdots\\[2pt]
		\frac{\ddd^0}{\ddd x^0}\big\vert_{x\mapsto 0}\frac{P_{(k+1)n+k\pm (m-1)}(x)}{0!} & \cdots & \frac{\ddd^{m-1}}{\ddd x^{m-1}}\big\vert_{x\mapsto 0} \frac{P_{(k+1)n+k\pm (m-1)}(x)}{(m-1)!}
	\end{pmatrix},
\end{align}
with $P_{(k+1)n+k\pm i}(x):= P_{(k+1)n+k+ i}(x) + P_{(k+1)n+k- i}(x)$. In view of \eqref{eq:F_n+-i},
\begin{align*}
	&\big(P_{(k+1)n+k}(x), P_{(k+1)n+k\pm 1}(x), \ldots, P_{(k+1)n+k\pm (m-1)}(x)\big)^{\mathsf{T}}\\
	&\qquad\qquad\qquad\qquad\qquad = P_{(k+1)n+k}(x) \big(1, L_1(x-c), \ldots, L_{m-1}(x-c)\big)^{\mathsf{T}}.
\end{align*}
As such, we can apply Lemma~\ref{le:det-derivative} and obtain
\begin{align*}
	\det M_P = P_{(k+1)n+k}(0)^m = \big((-1)^{(k+1)n+k}F_{(k+1)(n+1)}(c)\big)^m.
\end{align*}
Finally, we arrive at the desired identity
\begin{align*}
	\det_{0\leq i,j \leq N-1}\big(\Phi(x^{i+j+m} P_k(x))\big) = (-1)^{\binom{k+1}{2}n} \left(\frac{F_{(k+1)(n+1)}(c)}{F_{k+1}(c)}\right)^m.
\end{align*}

\subsubsection{Case $r= k-m+1$}\label{sec:b0m-b}

Let $k\ge m-1$. In the context of Corollary~\ref{coro:Kra-2}, we choose the permutation
\begin{align*}
	\sigma(a) = (k+m-1) - a,\qquad\qquad a\in [0,k+m-1],
\end{align*}
so that
\begin{align*}
	\sgn(\sigma) = (-1)^{\binom{k+m}{2}}.
\end{align*}
For $i\in [0,k-1]$, we have
\begin{align*}
	P_{N+\sigma(i)}(\rho_j) \overset{\eqref{eq:P-1}}{=} (-1)^{nj} P_{2k-i}(\rho_j) \overset{\eqref{eq:P-3}}{=} (-1)^{nj+1} P_{i}(\rho_j).
\end{align*}
In addition,
\begin{align*}
	P_{N+\sigma(k)}(\rho_j) \overset{\eqref{eq:P-1}}{=} (-1)^{nj} P_{k}(\rho_j) \overset{\eqref{eq:P-2}}{=} 0,
\end{align*}
and for $s\in [k+1, k+m-1]$,
\begin{align*}
	P_{N+\sigma(s)}(\rho_j) \overset{\eqref{eq:P-1}}{=} (-1)^{nj} P_{2k-s}(\rho_j) \overset{\eqref{eq:P-3}}{=} -(-1)^{nj} P_{s}(\rho_j) \overset{\eqref{eq:P-1}}{=} -P_{N+(-k+m-1+s)}(\rho_j).
\end{align*}
Thus, by \eqref{eq:Kra-2},
\begin{align*}
	\det_{0\leq i,j \leq N-1}\big(\Phi(x^{i+j+m} P_k(x))\big) = \sgn(\sigma) (-1)^{(k+m)N} \frac{\det M_P}{\big((-1)^kF_{k+1}(c)\big)^m} \prod_{j=1}^k (-1)^{nj+1},
\end{align*}
where $M_P$ is as in \eqref{eq:MP}. We then arrive at the desired identity
\begin{align*}
	\det_{0\leq i,j \leq N-1}\big(\Phi(x^{i+j+m} P_k(x))\big) = (-1)^{\binom{k+1}{2}n+\binom{k-m+1}{2}} \left(\frac{F_{(k+1)(n+1)}(c)}{F_{k+1}(c)}\right)^m.
\end{align*}

\section{Case $b=1$}\label{sec:b1}

Define a second family of polynomials
\begin{align}\label{eq:Q-def}
	Q_n(x) := F_{n+1}(x-c) - F_{n}(x-c),\qquad\qquad n\ge 0.
\end{align}
Again, $Q_n(x)$ is monic and of degree $n$. For an arbitrary polynomial $f(x)$ of degree $d$, there is also a \emph{unique} expansion
\begin{align}
	f(x) = \sum_{l=0}^d f_l^\ddagger Q_l(x).
\end{align}
We then define a \emph{linear} functional $\Psi$ on the ring of polynomials in $x$ by
\begin{align}
	\Psi(f) := f_0^\ddagger.
\end{align}

Now we establish the following evaluation of the linear functional $\Psi$.

\begin{lemma}\label{le:Psi-eva}
	For $n,k\ge 0$,
	\begin{align}\label{eq:Psi-eva}
		\Psi\big(x^n Q_k(x)\big) = \Cat_{n,k}^{(1,c)}.
	\end{align}
\end{lemma}

\begin{proof}
	By \eqref{eq:F-rec},
	\begin{align}\label{eq:Q-rec}
		Q_{k+1}(x) = (x-c) Q_k(x) - Q_{k-1}(x).
	\end{align}
	Prepending $Q_{-1}(x) := 0$ and rearranging the terms, we have
	\begin{align*}
		x Q_k(x) = Q_{k+1}(x) + s'_k Q_k(x) + Q_{k-1}(x),
	\end{align*}
	where
	\begin{align*}
		(s'_0,s'_1,s'_2,\ldots) := (c+1,c,c,\ldots).
	\end{align*}
	Thus,
	\begin{align*}
		\Psi\big(x^n Q_k(x)\big) = \Psi\big(x^{n-1} Q_{k+1}(x)\big) + s'_k \Psi\big(x^{n-1} Q_k(x)\big) + \Psi\big(x^{n-1} Q_{k-1}(x)\big).
	\end{align*}
	In addition, $\Psi\big(x^0 Q_k(x)\big) = \delta_{k,0}$. It follows that $\Psi\big(x^n Q_k(x)\big)$ satisfies the same recurrence and initial values as $\Cat_{n,k}^{(1,c)}$, thereby showing the claimed equality.
\end{proof}

The next orthogonality relation can be shown in the same way as that for Lemma~\ref{le:P-orthogonal}.

\begin{lemma}
	The monic polynomials $(Q_n(x))_{n\ge 0}$ are orthogonal associated with the linear functional $\Psi$.
\end{lemma}

Furthermore, the three-term recursive relation \eqref{eq:Q-rec} gives us the Hankel determinants for the moments $\Psi(x^n)$.

\begin{proposition}\label{prop:det-Psi}
	We have
	\begin{align}
		\underset{{0\le i,j\le N-1}}{\det}\big(\Psi(x^{i+j})\big) = 1.
	\end{align}
\end{proposition}

Finally, we factorize $Q_n(x)$.

\begin{lemma}\label{le:Q-factorization}
	For $n\ge 0$,
	\begin{align}\label{eq:Q-factorization}
		Q_n(x) = \prod_{j=1}^n (x-\varrho_{n,j}),
	\end{align}
	where
	\begin{align*}
		\varrho_{n,j} := c + 2\cos \tfrac{(2j-1) \pi}{2n+1},\qquad\qquad 1\le j\le n,
	\end{align*}
	are pairwise distinct.
\end{lemma}

\begin{proof}
	Let $0< \vartheta < \pi$. By \eqref{eq:F-eva},
	\begin{align}\label{eq:Q-vartheta}
		Q_n(c+2\cos\vartheta) &= F_{n+1}(2\cos \vartheta) - F_{n}(2\cos \vartheta)\notag\\
		& = \frac{\sin((n+1)\vartheta) - \sin (n\vartheta)}{\sin \vartheta}  = \frac{\cos \frac{(2n+1)\vartheta}{2}}{\cos \frac{\vartheta}{2}}.
	\end{align}
	Thus, the pairwise distinct numbers $\varrho_{n,j}$ are roots of $Q_n(x)$, thereby implying the claimed factorization.
\end{proof}

For the moment, we fix $k\ge 0$ and define
\begin{align}
	\varrho_{j} := c + 2\cos \tfrac{(2j-1) \pi}{2k+1},\qquad\qquad 1\le j\le k.
\end{align}
In light of \eqref{eq:Q-vartheta}, we find parallel properties to Lemma~\ref{le:P-property}.

\begin{lemma}\label{le:Q-property}
	Let $j$ with $1\le j\le k$ be arbitrary. For $r\ge 0$,
	\begin{align}\label{eq:Q-1}
		Q_{(2k+1)n+r}(\varrho_j) = (-1)^n Q_r(\varrho_j).
	\end{align}
	Also,
	\begin{align}\label{eq:Q-2}
		Q_k(\varrho_j) = 0,
	\end{align}
	and for $s$ with $1\le s\le k$,
	\begin{align}\label{eq:Q-3}
		Q_{k+s}(\varrho_j) = -Q_{k-s}(\varrho_j).
	\end{align}
\end{lemma}

Now we are ready to evaluate our Hankel determinants. By Lemma~\ref{le:Psi-eva},
\begin{align}
	\det_{0\leq i,j \leq N-1}\big(\Cat_{i+j+m,k}^{(1,c)}\big) = \det_{0\leq i,j \leq N-1}\big(\Psi(x^{i+j+m} Q_k(x))\big).
\end{align} 
Hence, we take $p = Q$ and $L = \Psi$ in Corollaries~\ref{coro:Kra-1} and \ref{coro:Kra-2}. Also, we choose
\begin{align*}
	u(x) = Q_k(x) = (x-\varrho_1)\cdots (x-\varrho_k),
\end{align*}
with the latter equality coming from Lemma~\ref{le:Q-factorization}. Moreover, $Q_k(0)$, as a function of $c$, is never identical to the zero function when $k\ge 0$. Finally, in Corollary~\ref{coro:Kra-2}, for the denominator of the left-hand side of \eqref{eq:Kra-2}, we recall from Proposition~\ref{prop:det-Psi} that
\begin{align*}
	\underset{{0\le i,j\le N-1}}{\det}\big(\Psi(x^{i+j})\big) = 1.
\end{align*}
In the rest of this section, we always write
\begin{align*}
	N=(2k+1)n+r
\end{align*}
with $0\le r\le 2k$.

\subsection{Case $m=0$}

In this part, we prove Theorem~\ref{th:b1-m0}.

\subsubsection{Case $r\ne 0,k+1$}

Let $k\ge 1$ to make the choice of $r$ not vacuous. In the context of Corollary~\ref{coro:Kra-1}, we put $m=0$. If $1\le r\le k$,
\begin{align*}
	Q_{N+(k-r)}(\varrho_j) \overset{\eqref{eq:Q-1}}{=} (-1)^n Q_k(\varrho_j) \overset{\eqref{eq:Q-2}}{=} 0,
\end{align*}
with $k-r \in [0,k-1]$. If $k+2\le r\le 2k$,
\begin{align*}
	Q_{N+(2k-r+1)}(\varrho_j) \overset{\eqref{eq:Q-1}}{=} (-1)^{n+1} Q_{0}(\varrho_j) \overset{\eqref{eq:Q-3}}{=} (-1)^n Q_{2k}(\varrho_{j}) \overset{\eqref{eq:Q-1}}{=} Q_{N+(2k-r)}(\varrho_j),
\end{align*}
with $2k-r+1\in [0,k-1]$ and $2k-r\in [0,k-1]\backslash\{2k-r+1\}$. The determinant in question then vanishes.

\subsubsection{Case $r= 0$}

Let $k\ge 0$. In the context of Corollary~\ref{coro:Kra-2}, we put $m=0$. Choose the permutation
\begin{align*}
	\sigma(a) = a,\qquad\qquad a\in [0,k-1],
\end{align*}
so that
\begin{align*}
	\sgn(\sigma) = 1.
\end{align*}
For $i\in [0,k-1]$, we have
\begin{align*}
	Q_{N+\sigma(i)}(\varrho_j) \overset{\eqref{eq:Q-1}}{=} (-1)^{n} Q_{i}(\varrho_j).
\end{align*}
Thus, by \eqref{eq:Kra-2},
\begin{align*}
	\det_{0\leq i,j \leq N-1}\big(\Psi(x^{i+j} Q_k(x))\big) = \sgn(\sigma) (-1)^{kN} \prod_{j=1}^k (-1)^{n} = 1,
\end{align*}
as desired.

\subsubsection{Case $r= k+1$}

Let $k\ge 1$. In the context of Corollary~\ref{coro:Kra-2}, we put $m=0$. Choose the permutation
\begin{align*}
	\sigma(a) = (k-1) - a,\qquad\qquad a\in [0,k-1],
\end{align*}
so that
\begin{align*}
	\sgn(\sigma) = (-1)^{\binom{k}{2}}.
\end{align*}
For $i\in [0,k-1]$, we have
\begin{align*}
	Q_{N+\sigma(i)}(\varrho_j) \overset{\eqref{eq:Q-1}}{=} (-1)^{n} Q_{2k-i}(\varrho_j) \overset{\eqref{eq:Q-3}}{=} (-1)^{n+1} Q_{i}(\varrho_j).
\end{align*}
Thus, by \eqref{eq:Kra-2},
\begin{align*}
	\det_{0\leq i,j \leq N-1}\big(\Psi(x^{i+j} Q_k(x))\big) = \sgn(\sigma) (-1)^{kN} \prod_{j=1}^k (-1)^{n+1} = (-1)^{\binom{k+1}{2}},
\end{align*}
as desired.

\subsection{Case $m=1$}

In this part, we prove Theorem~\ref{th:b1-m1}.

\subsubsection{Case $r\ne 0,k,k+1,2k$}

Let $k\ge 2$. In the context of Corollary~\ref{coro:Kra-1}, we put $m=1$. If $1\le r\le k-1$,
\begin{align*}
	Q_{N+(k-r)}(\varrho_j) \overset{\eqref{eq:Q-1}}{=} (-1)^n Q_k(\varrho_j) \overset{\eqref{eq:Q-2}}{=} 0,
\end{align*}
and
\begin{align*}
	Q_{N+(k-r+1)}(\varrho_j) &\overset{\eqref{eq:Q-1}}{=} (-1)^{n} Q_{k+1}(\varrho_j)\\
	& \overset{\eqref{eq:Q-3}}{=} -(-1)^{n} Q_{k-1}(\varrho_j) \overset{\eqref{eq:Q-1}}{=} -Q_{N+(k-r-1)}(\varrho_j),
\end{align*}
with $S:=\{k-r,k-r+1\}\subset [0,k]$ and $\{k-r-1\}\subset [0,k]\backslash S$. If $k+2\le r\le 2k-1$,
\begin{align*}
	Q_{N+(2k-r+1)}(\varrho_j) \overset{\eqref{eq:Q-1}}{=} (-1)^{n+1} Q_{0}(\varrho_j) \overset{\eqref{eq:Q-3}}{=} (-1)^n Q_{2k}(\varrho_{j}) \overset{\eqref{eq:Q-1}}{=} Q_{N+(2k-r)}(\varrho_j),
\end{align*}
and
\begin{align*}
	Q_{N+(2k-r+2)}(\varrho_j) &\overset{\eqref{eq:Q-1}}{=} (-1)^{n+1} Q_{1}(\varrho_j)\\
	& \overset{\eqref{eq:Q-3}}{=} (-1)^n Q_{2k-1}(\varrho_{j}) \overset{\eqref{eq:Q-1}}{=} Q_{N+(2k-r-1)}(\varrho_j),
\end{align*}
with $S:=\{2k-r+1,2k-r+2\}\subset [0,k]$ and $\{2k-r,2k-r-1\}\subset [0,k]\backslash S$. The determinant in question then vanishes.

\subsubsection{Cases $r= 0$, $k$, $k+1$, and $2k$}

These correspond to the $m=1$ case of Theorem~\ref{th:b1-m}, and we defer its proof to Section~\ref{sec:b1m}.

\subsection{Case $m=2$}

In this part, we prove Theorem~\ref{th:b1-m2}.

\subsubsection{Case $r\ne 0,k,k\pm 1,2k,2k-1$}

Let $k\ge 3$. In the context of Corollary~\ref{coro:Kra-1}, we put $m=1$. If $1\le r\le k-2$,
\begin{align*}
	Q_{N+(k-r)}(\varrho_j) \overset{\eqref{eq:Q-1}}{=} (-1)^n Q_k(\varrho_j) \overset{\eqref{eq:Q-2}}{=} 0,
\end{align*}
and
\begin{align*}
	Q_{N+(k-r+1)}(\varrho_j) &\overset{\eqref{eq:Q-1}}{=} (-1)^{n} Q_{k+1}(\varrho_j)\\
	& \overset{\eqref{eq:Q-3}}{=} -(-1)^{n} Q_{k-1}(\varrho_j) \overset{\eqref{eq:Q-1}}{=} -Q_{N+(k-r-1)}(\varrho_j),
\end{align*}
and
\begin{align*}
	Q_{N+(k-r+2)}(\varrho_j) &\overset{\eqref{eq:Q-1}}{=} (-1)^{n} Q_{k+2}(\varrho_j)\\
	& \overset{\eqref{eq:Q-3}}{=} -(-1)^{n} Q_{k-2}(\varrho_j) \overset{\eqref{eq:Q-1}}{=} -Q_{N+(k-r-2)}(\varrho_j),
\end{align*}
with $S:=\{k-r,k-r+1,k-r+2\}\subset [0,k+1]$ and $\{k-r-1,k-r-2\}\subset [0,k+1]\backslash S$. If $k+2\le r\le 2k-2$,
\begin{align*}
	Q_{N+(2k-r+1)}(\varrho_j) \overset{\eqref{eq:Q-1}}{=} (-1)^{n+1} Q_{0}(\varrho_j) \overset{\eqref{eq:Q-3}}{=} (-1)^n Q_{2k}(\varrho_{j}) \overset{\eqref{eq:Q-1}}{=} Q_{N+(2k-r)}(\varrho_j),
\end{align*}
and
\begin{align*}
	Q_{N+(2k-r+2)}(\varrho_j) &\overset{\eqref{eq:Q-1}}{=} (-1)^{n+1} Q_{1}(\varrho_j)\\
	& \overset{\eqref{eq:Q-3}}{=} (-1)^n Q_{2k-1}(\varrho_{j}) \overset{\eqref{eq:Q-1}}{=} Q_{N+(2k-r-1)}(\varrho_j),
\end{align*}
and
\begin{align*}
	Q_{N+(2k-r+3)}(\varrho_j) &\overset{\eqref{eq:Q-1}}{=} (-1)^{n+1} Q_{2}(\varrho_j)\\
	& \overset{\eqref{eq:Q-3}}{=} (-1)^n Q_{2k-2}(\varrho_{j}) \overset{\eqref{eq:Q-1}}{=} Q_{N+(2k-r-2)}(\varrho_j),
\end{align*}
with $S:=\{2k-r+1,2k-r+2,2k-r+3\}\subset [0,k+1]$ and $\{2k-r,2k-r-1,2k-r-2\}\subset [0,k+1]\backslash S$. The determinant in question then vanishes.

\subsubsection{Cases $r= 0$, $k-1$, $k+1$, and $2k-1$}

These correspond to the $m=2$ case of Theorem~\ref{th:b1-m}, and we defer its proof to Section~\ref{sec:b1m}.

\subsubsection{Case $r= k$ or $2k$}

Here our objective is to show \eqref{eq:b1-m2-sum}. We start by establishing individual evaluations with respect to the two residue classes in question.

\begin{theorem}\label{th:b1-rk}
	Let $k\ge 1$ and write $N=(2k+1)n+k$. Then for $c$ an indeterminate,
	\begin{align}\label{eq:b1-rk}
		\det_{0\leq i,j \leq N-1}\big(\Cat_{i+j+2,k}^{(1,c)}\big) = (-1)^{\binom{k+1}{2}}\frac{\det W}{Q_k(0)^2},
	\end{align}
	where
	\begin{align*}
		W := \left.\begin{psmallmatrix}
			Q_{(2k+1)n+k}(x) & \frac{\ddd}{\ddd x}Q_{(2k+1)n+k}(x)\\[2pt]
			Q_{(2k+1)n+2k+1}(x) - Q_{(2k+1)n+2k}(x) & \frac{\ddd}{\ddd x}\big(Q_{(2k+1)n+2k+1}(x) - Q_{(2k+1)n+2k}(x)\big)
		\end{psmallmatrix}\right\vert_{x\mapsto 0}.
	\end{align*}
\end{theorem}

\begin{proof}
	In the context of Corollary~\ref{coro:Kra-2}, we put $m=2$. Choose the permutation
	\begin{align*}
		\sigma(a) = \begin{cases}
			k - a,& \qquad\qquad a\in [0,k],\\
			k+1,& \qquad\qquad a=k+1,
		\end{cases}
	\end{align*}
	so that
	\begin{align*}
		\sgn(\sigma) = (-1)^{\binom{k+1}{2}}.
	\end{align*}
	For $i\in [0,k-1]$, we have
	\begin{align*}
		Q_{N+\sigma(i)}(\varrho_j) \overset{\eqref{eq:Q-1}}{=} (-1)^{n} Q_{2k-i}(\varrho_j) \overset{\eqref{eq:Q-3}}{=} (-1)^{n+1} Q_{i}(\varrho_j).
	\end{align*}
	In addition,
	\begin{align*}
		Q_{N+\sigma(k)}(\varrho_j) \overset{\eqref{eq:Q-1}}{=} (-1)^{n} Q_{k}(\varrho_j) \overset{\eqref{eq:Q-2}}{=} 0,
	\end{align*}
	and
	\begin{align*}
		Q_{N+\sigma(k+1)}(\varrho_j) \overset{\eqref{eq:Q-1}}{=} (-1)^{n+1} Q_{0}(\varrho_j) \overset{\eqref{eq:Q-3}}{=} (-1)^{n} Q_{2k}(\varrho_j) \overset{\eqref{eq:Q-1}}{=} Q_{N+k}(\varrho_j).
	\end{align*}
	Thus, by \eqref{eq:Kra-2},
	\begin{align*}
		\det_{0\leq i,j \leq N-1}\big(\Psi(x^{i+j+2} Q_k(x))\big) = \sgn(\sigma) (-1)^{(k+2)N} \frac{\det W}{Q_k(0)^2} \prod_{j=1}^k (-1)^{n+1},
	\end{align*}
	where $W$ is as in Theorem~\ref{th:b1-rk}. We then arrive at the desired identity.
\end{proof}

\begin{theorem}\label{th:b1-r2k}
	Let $k\ge 1$ and write $N=(2k+1)n+2k$. Then for $c$ an indeterminate,
	\begin{align}\label{eq:b1-r2k}
		\det_{0\leq i,j \leq N-1}\big(\Cat_{i+j+2,k}^{(1,c)}\big) = \frac{\det W'}{Q_k(0)^2},
	\end{align}
	where
	\begin{align*}
		W' := \left.\begin{psmallmatrix}
			Q_{(2k+1)n+3k+1}(x) & \frac{\ddd}{\ddd x}Q_{(2k+1)n+3k+1}(x)\\[2pt]
			Q_{(2k+1)n+2k+1}(x) - Q_{(2k+1)n+2k}(x) & \frac{\ddd}{\ddd x}\big(Q_{(2k+1)n+2k+1}(x) - Q_{(2k+1)n+2k}(x)\big)
		\end{psmallmatrix}\right\vert_{x\mapsto 0}.
	\end{align*}
\end{theorem}

\begin{proof}
	In the context of Corollary~\ref{coro:Kra-2}, we put $m=2$. Choose the permutation
	\begin{align*}
		\sigma(a) = \begin{cases}
			a+1,& \qquad\qquad a\in [0,k],\\
			0,& \qquad\qquad a=k+1,
		\end{cases}
	\end{align*}
	so that
	\begin{align*}
		\sgn(\sigma) = (-1)^{k+1}.
	\end{align*}
	For $i\in [0,k-1]$, we have
	\begin{align*}
		Q_{N+\sigma(i)}(\varrho_j) \overset{\eqref{eq:Q-1}}{=} (-1)^{n+1} Q_{i}(\varrho_j).
	\end{align*}
	In addition,
	\begin{align*}
		Q_{N+\sigma(k)}(\varrho_j) \overset{\eqref{eq:Q-1}}{=} (-1)^{n+1} Q_{k}(\varrho_j) \overset{\eqref{eq:Q-2}}{=} 0,
	\end{align*}
	and
	\begin{align*}
		Q_{N+\sigma(k+1)}(\varrho_j) \overset{\eqref{eq:Q-1}}{=} (-1)^{n} Q_{2k}(\varrho_j) \overset{\eqref{eq:Q-3}}{=} (-1)^{n+1} Q_{0}(\varrho_j) \overset{\eqref{eq:Q-1}}{=} Q_{N+1}(\varrho_j).
	\end{align*}
	Thus, by \eqref{eq:Kra-2},
	\begin{align*}
		\det_{0\leq i,j \leq N-1}\big(\Psi(x^{i+j+2} Q_k(x))\big) = \sgn(\sigma) (-1)^{(k+2)N} \frac{\det \tilde{W}'}{Q_k(0)^2} \prod_{j=1}^k (-1)^{n+1},
	\end{align*}
	where
	\begin{align*}
		\tilde{W}' = \left.\begin{psmallmatrix}
			Q_{(2k+1)n+3k+1}(x) & \frac{\ddd}{\ddd x}Q_{(2k+1)n+3k+1}(x)\\[2pt]
			Q_{(2k+1)n+2k}(x) - Q_{(2k+1)n+2k+1}(x) & \frac{\ddd}{\ddd x}\big(Q_{(2k+1)n+2k}(x) - Q_{(2k+1)n+2k+1}(x)\big)
		\end{psmallmatrix}\right\vert_{x\mapsto 0}.
	\end{align*}
	Note that $\det \tilde{W}' = -\det W'$ by multiplying the second row by $-1$. We then arrive at the desired identity.
\end{proof}

In light of \eqref{eq:b1-rk} and \eqref{eq:b1-r2k}, we have
\begin{align*}
	\LHS\eqref{eq:b1-m2-sum} = \frac{\det W^\dagger}{Q_k(0)^2},
\end{align*}
where
\begin{align*}
	W^\dagger := \left.\begin{psmallmatrix}
		Q_{(2k+1)n+3k+1}(x) - Q_{(2k+1)n+k}(x) & \frac{\ddd}{\ddd x}\big(Q_{(2k+1)n+3k+1}(x) - Q_{(2k+1)n+k}(x)\big)\\[2pt]
		Q_{(2k+1)n+2k+1}(x) - Q_{(2k+1)n+2k}(x) & \frac{\ddd}{\ddd x}\big(Q_{(2k+1)n+2k+1}(x) - Q_{(2k+1)n+2k}(x)\big)
	\end{psmallmatrix}\right\vert_{x\mapsto 0}.
\end{align*}
It follows from \eqref{eq:D-3}, in which \eqref{eq:F-L} should be applied, that
\begin{align*}
	\det M^\dagger &= Q_k(0)^2 \big(L_{k+1}(-c) - L_k(-c)\big)^2 \left(\frac{F_{(2k+1)(n+1)}(-c)}{F_{2k+1}(-c)}\right)^2\\
	&\quad\times \det\left.\begin{pmatrix}
		F_{k+1}(x-c) + F_k(x-c) & \frac{\ddd}{\ddd x} \big(F_{k+1}(x-c) + F_k(x-c)\big)\\
		F_{1}(x-c) + F_0(x-c) & \frac{\ddd}{\ddd x} \big(F_{1}(x-c) + F_0(x-c)\big)
	\end{pmatrix}\right\vert_{x\mapsto 0}.
\end{align*}
The determinant on the right-hand side of the above equals
\begin{align*}
	-\big(F'_{k+1}(-c)+F'_k(-c)\big) = (-1)^k \big(F'_{k+1}(c)-F'_k(c)\big),
\end{align*}
where we have applied \eqref{eq:F'(-x)}. Meanwhile, by \eqref{eq:F(-x)} and \eqref{eq:L(-x)},
\begin{align*}
	&\big(L_{k+1}(-c) - L_k(-c)\big)^2 \left(\frac{F_{(2k+1)(n+1)}(-c)}{F_{2k+1}(-c)}\right)^2\\
	&\qquad\qquad =\big(L_{k+1}(c) + L_k(c)\big)^2 \left(\frac{F_{(2k+1)(n+1)}(c)}{F_{2k+1}(c)}\right)^2.
\end{align*}
We conclude that
\begin{align*}
	\det M^\dagger = (-1)^k \big(L_{k+1}(c)+L_{k}(c)\big)^2 \big(F'_{k+1}(c)-F'_k(c)\big)\left(\frac{F_{(2k+1)(n+1)}(c)}{F_{2k+1}(c)}\right)^2 Q_k(0)^2,
\end{align*}
thereby yielding \eqref{eq:b1-m2-sum}.

\subsection{Generic $m$}\label{sec:b1m}

In this part, we prove Theorem~\ref{th:b1-m}. Throughout, fix $m\ge 1$.

\subsubsection{Case $r= 0$}\label{sec:b1m-a}

Let $k\ge m-1$. In the context of Corollary~\ref{coro:Kra-2}, we choose the permutation
\begin{align*}
	\sigma(a) = a,\qquad\qquad a\in [0,k+m-1],
\end{align*}
so that
\begin{align*}
	\sgn(\sigma) = 1.
\end{align*}
For $i\in [0,k-1]$, we have
\begin{align*}
	Q_{N+\sigma(i)}(\varrho_j) \overset{\eqref{eq:Q-1}}{=} (-1)^{n} Q_{i}(\varrho_j).
\end{align*}
In addition,
\begin{align*}
	Q_{N+\sigma(k)}(\varrho_j) \overset{\eqref{eq:Q-1}}{=} (-1)^{n} Q_{k}(\varrho_j) \overset{\eqref{eq:Q-2}}{=} 0,
\end{align*}
and for $s$ with $k+1\le s\le k+m-1$,
\begin{align*}
	Q_{N+\sigma(s)}(\varrho_j) \overset{\eqref{eq:Q-1}}{=} (-1)^{n} Q_{s}(\varrho_j) \overset{\eqref{eq:Q-3}}{=} -(-1)^{n} Q_{2k-s}(\varrho_j) \overset{\eqref{eq:Q-1}}{=} -Q_{N+(2k-s)}(\varrho_j).
\end{align*}
Thus, by \eqref{eq:Kra-2},
\begin{align*}
	\det_{0\leq i,j \leq N-1}\big(\Psi(x^{i+j+m} Q_k(x))\big) = \sgn(\sigma) (-1)^{(k+m)N} \frac{\det M_Q}{Q_k(0)^m} \prod_{j=1}^k (-1)^{n},
\end{align*}
where
\begin{align}\label{eq:MQ}
	M_Q := \begin{pmatrix}
		\frac{\ddd^0}{\ddd x^0}\big\vert_{x\mapsto 0}\frac{Q_{(2k+1)n+k}(x)}{0!} & \cdots & \frac{\ddd^{m-1}}{\ddd x^{m-1}}\big\vert_{x\mapsto 0} \frac{Q_{(2k+1)n+k}(x)}{(m-1)!}\\[2pt]
		\frac{\ddd^0}{\ddd x^0}\big\vert_{x\mapsto 0}\frac{Q_{(2k+1)n+k\pm 1}(x)}{0!} & \cdots & \frac{\ddd^{m-1}}{\ddd x^{m-1}}\big\vert_{x\mapsto 0} \frac{Q_{(2k+1)n+k\pm 1}(x)}{(m-1)!}\\[2pt]
		\vdots & \ddots & \vdots\\[2pt]
		\frac{\ddd^0}{\ddd x^0}\big\vert_{x\mapsto 0}\frac{Q_{(2k+1)n+k\pm (m-1)}(x)}{0!} & \cdots & \frac{\ddd^{m-1}}{\ddd x^{m-1}}\big\vert_{x\mapsto 0} \frac{Q_{(2k+1)n+k\pm (m-1)}(x)}{(m-1)!}
	\end{pmatrix},
\end{align}
with $Q_{(2k+1)n+k\pm i}(x):= Q_{(2k+1)n+k+ i}(x) + Q_{(2k+1)n+k- i}(x)$. In view of \eqref{eq:F_n+-i},
\begin{align*}
	&\big(Q_{(k+1)n+k}(x), Q_{(k+1)n+k\pm 1}(x), \ldots, Q_{(k+1)n+k\pm (m-1)}(x)\big)^{\mathsf{T}}\\
	&\qquad\qquad\qquad\qquad\qquad = Q_{(k+1)n+k}(x) \big(1, L_1(x-c), \ldots, L_{m-1}(x-c)\big)^{\mathsf{T}}.
\end{align*}
As such, we can apply Lemma~\ref{le:det-derivative} and obtain
\begin{align*}
	\det M_Q &= Q_{(2k+1)n+k}(0)^m = \left(\frac{F_{(2k+1)(n+1)}(-c)}{F_{2k+1}(-c)}-\frac{F_{(2k+1)n}(-c)}{F_{2k+1}(-c)}\right)^m Q_k(0)^m\\
	&= (-1)^{mn} \left(\frac{F_{(2k+1)(n+1)}(c)}{F_{2k+1}(c)}+\frac{F_{(2k+1)n}(c)}{F_{2k+1}(c)}\right)^m Q_k(0)^m,
\end{align*}
where we have further applied \eqref{eq:D-1}, \eqref{eq:F-L}, and \eqref{eq:F(-x)}. Finally, we arrive at the desired identity
\begin{align*}
	\det_{0\leq i,j \leq N-1}\big(\Psi(x^{i+j+m} Q_k(x))\big) = \left(\frac{F_{(2k+1)n}(c)}{F_{2k+1}(c)} + \frac{F_{(2k+1)(n+1)}(c)}{F_{2k+1}(c)}\right)^m.
\end{align*}

\subsubsection{Case $r= k-m+1$}\label{sec:b1m-b}

Let $k\ge m-1$. In the context of Corollary~\ref{coro:Kra-2}, we choose the permutation
\begin{align*}
	\sigma(a) = (k+m-1) - a,\qquad\qquad a\in [0,k+m-1],
\end{align*}
so that
\begin{align*}
	\sgn(\sigma) = (-1)^{\binom{k+m}{2}}.
\end{align*}
For $i\in [0,k-1]$, we have
\begin{align*}
	Q_{N+\sigma(i)}(\varrho_j) \overset{\eqref{eq:Q-1}}{=} (-1)^{n} Q_{2k-i}(\varrho_j) \overset{\eqref{eq:Q-3}}{=} (-1)^{n-1} Q_{i}(\varrho_j).
\end{align*}
In addition,
\begin{align*}
	Q_{N+\sigma(k)}(\varrho_j) \overset{\eqref{eq:Q-1}}{=} (-1)^{n} Q_{k}(\varrho_j) \overset{\eqref{eq:Q-2}}{=} 0,
\end{align*}
and for $s$ with $k+1\le s\le k+m-1$,
\begin{align*}
	Q_{N+\sigma(s)}(\varrho_j) \overset{\eqref{eq:Q-1}}{=} (-1)^{n} Q_{2k-s}(\varrho_j) \overset{\eqref{eq:Q-3}}{=} -(-1)^{n} Q_{s}(\varrho_j) \overset{\eqref{eq:Q-1}}{=} -Q_{N+(-k+m-1+s)}(\varrho_j).
\end{align*}
Thus, by \eqref{eq:Kra-2},
\begin{align*}
	\det_{0\leq i,j \leq N-1}\big(\Psi(x^{i+j+m} Q_k(x))\big) = \sgn(\sigma) (-1)^{(k+m)N} \frac{\det M_Q}{Q_k(0)^m} \prod_{j=1}^k (-1)^{n-1},
\end{align*}
where $M_Q$ is as in \eqref{eq:MQ}. We then arrive at the desired identity
\begin{align*}
	\det_{0\leq i,j \leq N-1}\big(\Psi(x^{i+j+m} Q_k(x))\big) = (-1)^{\binom{k-m+1}{2}}\left(\frac{F_{(2k+1)n}(c)}{F_{2k+1}(c)} + \frac{F_{(2k+1)(n+1)}(c)}{F_{2k+1}(c)}\right)^m.
\end{align*}

\subsubsection{Case $r=k+1$}\label{sec:b1m-c}

Let $k\ge m$. In the context of Corollary~\ref{coro:Kra-2}, we choose the permutation
\begin{align*}
	\sigma(a) = \begin{cases}
		(k-1)-a, & \qquad\qquad a\in [0,k-1],\\
		a, & \qquad\qquad a\in [k,k+m-1],
	\end{cases}
\end{align*}
so that
\begin{align*}
	\sgn(\sigma) = (-1)^{\binom{k}{2}}.
\end{align*}
For $i\in [0,k-1]$, we have
\begin{align*}
	Q_{N+\sigma(i)}(\varrho_j) \overset{\eqref{eq:Q-1}}{=} (-1)^{n} Q_{2k-i}(\varrho_j) \overset{\eqref{eq:Q-3}}{=} (-1)^{n-1} Q_{i}(\varrho_j).
\end{align*}
In addition, for $s\in [k,k+m-1]$,
\begin{align*}
	Q_{N+\sigma(s)}(\varrho_j) \overset{\eqref{eq:Q-1}}{=} (-1)^{n+1} Q_{s-k}(\varrho_j) \overset{\eqref{eq:Q-3}}{=} (-1)^{n} Q_{3k-s}(\varrho_j) \overset{\eqref{eq:Q-1}}{=} Q_{N+(2k-s-1)}(\varrho_j).
\end{align*}
Thus, by \eqref{eq:Kra-2},
\begin{align*}
	\det_{0\leq i,j \leq N-1}\big(\Psi(x^{i+j+2} Q_k(x))\big)= \sgn(\sigma) (-1)^{(k+m)N} \frac{\det M'_Q}{Q_k(0)^m} \prod_{j=1}^k (-1)^{n-1},
\end{align*}
where
\begin{align}\label{eq:M'Q}
	M'_Q := \left(\frac{\ddd^j}{\ddd x^j}\bigg\vert_{x\mapsto 0}\big(Q_{(2k+1)n+2k+1+i}(x)-Q_{(2k+1)n+2k-i}(x)\big)\right).
\end{align}
In view of \eqref{eq:D-3},
\begin{align*}
	&\big(Q_{(2k+1)n+2k+1+i}(x)-Q_{(2k+1)n+2k-i}(x)\big)_{0\le i\le m-1}^{\mathsf{T}}\\
	&\quad = Q_k(x) \big(L_{k+1}(x-c) - L_k(x-c)\big) \tfrac{F_{(2k+1)(n+1)}(x-c)}{F_{2k+1}(x-c)} \big(F_{i+1}(x)+F_i(x)\big)_{0\le i\le m-1}^{\mathsf{T}}.
\end{align*}
As such, we can apply Lemma~\ref{le:det-derivative} and obtain
\begin{align*}
	\det M'_Q &= \big(L_{k+1}(-c) - L_k(-c)\big)^m \left(\frac{F_{(2k+1)(n+1)}(-c)}{F_{2k+1}(-c)}\right)^m Q_k(0)^m\\
	&= (-1)^{m(k+n+1)} \big(L_{k+1}(c) + L_k(c)\big)^m \left(\frac{F_{(2k+1)(n+1)}(c)}{F_{2k+1}(c)}\right)^m Q_k(0)^m.
\end{align*}
Finally, we arrive at the desired identity
\begin{align*}
	&\det_{0\leq i,j \leq N-1}\big(\Psi(x^{i+j+2} Q_k(x))\big)\\
	&\qquad = (-1)^{\binom{k+1}{2}}\big(L_{k+1}(c)+L_{k}(c)\big)^m\left(\frac{F_{(2k+1)(n+1)}(c)}{F_{2k+1}(c)}\right)^m.
\end{align*}

\subsubsection{Case $r=2k-m+1$}\label{sec:b1m-d}

Let $k\ge m$. In the context of Corollary~\ref{coro:Kra-2}, we choose the permutation
\begin{align*}
	\sigma(a) = \begin{cases}
		a+m, & \qquad\qquad a\in [0,k-1],\\
		(k+m-1) - a, & \qquad\qquad a\in [k,k+m-1],
	\end{cases}
\end{align*}
so that
\begin{align*}
	\sgn(\sigma) = (-1)^{\binom{m}{2}+mk}.
\end{align*}
For $i\in [0,k-1]$, we have
\begin{align*}
	Q_{N+\sigma(i)}(\varrho_j) \overset{\eqref{eq:Q-1}}{=} (-1)^{n+1} Q_{i}(\varrho_j).
\end{align*}
In addition, for $s\in [k,k+m-1]$,
\begin{align*}
	Q_{N+\sigma(s)}(\varrho_j) \overset{\eqref{eq:Q-1}}{=} (-1)^{n} Q_{3k-s}(\varrho_j) \overset{\eqref{eq:Q-3}}{=} (-1)^{n+1} Q_{s-k}(\varrho_j) \overset{\eqref{eq:Q-1}}{=} Q_{N+(-k+m+s)}(\varrho_j).
\end{align*}
Thus, by \eqref{eq:Kra-2},
\begin{align*}
	\det_{0\leq i,j \leq N-1}\big(\Psi(x^{i+j+2} Q_k(x))\big)= \sgn(\sigma) (-1)^{(k+m)N} \frac{\det M''_Q}{Q_k(0)^m} \prod_{j=1}^k (-1)^{n+1},
\end{align*}
where
\begin{align*}
	M''_Q := \left(\frac{\ddd^j}{\ddd x^j}\bigg\vert_{x\mapsto 0}\big(Q_{(2k+1)n+2k-i}(x)-Q_{(2k+1)n+2k+1+i}(x)\big)\right).
\end{align*}
Note that $M''_Q = -M'_Q$, where $M'_Q$ is as in \eqref{eq:M'Q}. Thus,
\begin{align*}
	\det M''_Q = (-1)^m \det M'_Q.
\end{align*}
We then arrive at the desired identity
\begin{align*}
	&\det_{0\leq i,j \leq N-1}\big(\Psi(x^{i+j+2} Q_k(x))\big)\\
	&\qquad = (-1)^{\binom{m}{2}+mk}\big(L_{k+1}(c)+L_{k}(c)\big)^m\left(\frac{F_{(2k+1)(n+1)}(c)}{F_{2k+1}(c)}\right)^m.
\end{align*}

\section{The binomial determinant of Cigler and Krattenthaler}\label{sec:binomial}

Now we prove the conjectural binomial determinant identity of Cigler and Krattenthaler stated in Theorem~\ref{th:CK}. We claim that
\begin{align}\label{eq:Cat-b1c2}
	\Cat_{n,k}^{(1,2)} = \binom{2n+1}{n+k+1}.
\end{align}
It is clear that these binomial coefficients satisfy the boundary values in \eqref{eq:a-boundary}. For $k=0$, applications of Pascal's identity give us
\begin{align*}
	\binom{2n+1}{n+1} &= \binom{2n}{n+1} + \binom{2n}{n}\\
	&= \left[\binom{2n-1}{n+1} + \binom{2n-1}{n}\right] + \left[\binom{2n-1}{n} + \binom{2n-1}{n-1}\right]\\
	&= 3\binom{2n-1}{n} + \binom{2n-1}{n+1},
\end{align*}
where we have used $\binom{2n-1}{n-1} = \binom{2n-1}{n}$. Similarly, for $k\ge 1$, we have
\begin{align*}
	\binom{2n+1}{n+k+1} = \binom{2n-1}{n+k-1} + 2\binom{2n-1}{n+k} + \binom{2n-1}{n+k+1}.
\end{align*}
Thus, the recurrence in \eqref{eq:a-rec} also holds, thereby implying the claimed relation \eqref{eq:Cat-b1c2}. Now,
\begin{align*}
	\binom{2n+3}{n+k+1} = \Cat_{n+1,k-1}^{(1,2)}.
\end{align*}
Finally, in Theorem~\ref{th:b1-m1}, we take $c=2$ and replace $k$ with $k-1$. Then the desired determinant evaluations in Theorem~\ref{th:CK} hold by noting from the Binet-type formulas \eqref{eq:Binet-F} and \eqref{eq:Binet-L} that $F_n(2) = n$ and $L_n(2) = 2$.

\subsection*{Acknowledgements}

Shane Chern was supported by the Austrian Science Fund (No.~10.55776/F1002). Wenle Shi was supported by the China Scholarship Council (No.~202506060077).

\bibliographystyle{amsplain}

\end{document}